\documentclass[a4paper,11pt]{amsart}
\usepackage{amssymb,amscd,amsfonts,amsbsy,nicefrac}
\usepackage{amsmath,amsthm,amssymb,amsfonts,enumitem,verbatim}
\usepackage{amsrefs}
\usepackage{xcolor}
\usepackage{mathrsfs}
\usepackage{graphicx}

\calclayout

\usepackage{enumitem}
\usepackage{parskip}
\makeatletter
\g@addto@macro\bfseries{\boldmath}
\makeatother

\makeatletter
\def\subsubsection{\@startsection{subsubsection}{3}%
  \z@{.5\linespacing\@plus.7\linespacing}{-.5em}%
  {\normalfont\bf}}

\newtheorem{thm}{Theorem}[section]
\newtheorem{lem}[thm]{Lemma}
\newtheorem{lemma}[thm]{Lemma}
\newtheorem{cor}[thm]{Corollary}
\newtheorem{prop}[thm]{Proposition}
\newtheorem{defn}[thm]{Definition}
\newtheorem{rem}[thm]{Remark}

\def\vth{\vartheta}
\def\th{\theta}
\def\psiF{\psi}
\def\phiF{\phi}
\def\omegaF{\omega}

\def\R{\mathbb{R}}

\def\C{{\mathbb C}}

\def\2q{{\frac{2}{|B|}}}

\newcommand{\N}{\mathbb{N}}

\newcommand{\bmo}{\mathrm{bmo}}
\newcommand{\BMO}{\mathrm{BMO}}

\newcommand{\Z}{\mathbb{Z}}

\newcommand{\MM}{{\mathcal{M}}} 
\newcommand{\NN}{{\mathcal{N}}} 
\newcommand{\NNN}{{\mathscr{N}}} 
\newcommand{\RR}{{\mathscr{R}}}

\newcommand{\supp}{\mbox{supp}\,}

\renewcommand{\leq}{\leqslant}
\renewcommand{\geq}{\geqslant}

\def\bra#1{{\langle{#1}\rangle}}

\newcommand{\phase}{\varphi}
\newcommand{\Phase}{\Phi}

\newcommand{\Mod}[1]{\ (\mathrm{mod}\ #1)}

\newcommand{\m}[1]{\begin{equation*}
#1
\end{equation*}}
\newcommand{\nm}[2]{\begin{equation}\label{#1}
#2
\end{equation}}
\newcommand{\ma}[1]{\begin{align*}
#1
\end{align*}}
\newcommand{\nma}[2]{\begin{equation}\label{#1}
\begin{aligned}
#2
\end{aligned}
\end{equation}}

\newcommand{\abs}[1]{\left|#1\right|}
\newcommand{\set}[1]{\left\{#1\right\}}
\newcommand{\brkt}[1]{\left(#1\right)}
\newcommand{\dd}{\,\mathrm{d}}
\newcommand{\dddd}{\mathrm{d}}
\newcommand{\ddd}{\,\text{\rm{\mbox{\dj}}}}
\renewcommand{\d}{\partial}

\newcommand{\norm}[1]{\left\Vert#1\right\Vert}

\begin{document}

\title[Boundedness of multilinear Fourier integral operators] {Global boundedness of multilinear Fourier integral operators}
\author[S.~Rodr\'iguez-L\'opez]{Salvador Rodr\'iguez-L\'opez}
\address{Department of Mathematics, Stockholm University, SE-106 91 Stockholm, Sweden}
\email{s.rodriguez-lopez@math.su.se}
\author[D.~Rule]{David Rule}
\address{Mathematics Institute, Link\"oping University, SE-581 83 Link\"oping, Sweden}
\email{david.rule@liu.se}
\author[W.~Staubach]{Wolfgang Staubach}
\address{Department of Mathematics, Uppsala University, SE-751 06 Uppsala, Sweden}
\email{wulf@math.uu.se}
\date{\today}
\thanks{The first author has been partially supported by the Grant MTM2010-14946. The third author is partially supported by a grant from the Holger and Anna Crafoord foundation}

\begin{abstract}
We establish global regularity of multilinear Fourier integral operators that are associated to nonlinear wave equations on product of $L^p$ spaces by proving endpoint boundedness on suitable products spaces containing combinations of the local Hardy space, the local BMO and the $L^2$ spaces. 
\end{abstract}

\maketitle

\section{Introduction}\label{sec:intro}

This paper deals with the global boundedness of a class of multilinear Fourier integral operators that appear frequently in connection to nonlinear wave equations. To illustrate this fix a
smooth, compactly supported multilinear symbol $m$ on $\R^n$.
Let $T_m$ denote the multilinear paraproduct
\begin{equation}
T_m( f_1, \dots,f_N)(x):=\int_{\R^{nN}}  m(\Xi)\prod_{j=1}^N\left(\widehat{f}_j(\xi_j)e^{ix\cdot\xi_j}\right) \dd\Xi,
\end{equation}
where $\xi_j\in\R^n$ ($j=1,\dots,N$) and $\Xi=(\xi_1,\dots,\xi_N)\in \R^{nN}$.

Furthermore, let
\[
    \sqrt{-\Delta}\, f(x)= \int_{\R^n} |\xi| \,\widehat{f}(\xi)\,e^{ix\cdot\xi}\, \ddd\xi,
\]
where $\ddd \xi$ denotes the normalised Lebesgue measure ${\dd \xi}/{(2\pi)^n}$.
Consider now the wave equations
\begin{equation}
\label{eq:dispersive}
\left\{ \begin{array}{l} i\partial_t u +  \sqrt{-\Delta}\, u = T_{m}\left( v_1,\dots, v_N\right)  \\
i\partial_t v_k +  \sqrt{-\Delta}\, v_k = 0,\,\,\, k=1,\dots, N \\
\end{array} \right.
\quad \mbox{with} \quad
\left\{ \begin{array}{l} u(0,x) = 0  \\ v_k(0,x) = f_k (x), \,\,\, k=1,\dots, N. \end{array} \right.
\end{equation}
The functions $u$ and $v_k$ are complex valued, and each $f_k$ maps $\mathbb{R}^n$ to $\mathbb{C}$. The above system is used in order to study the nonlinear interaction of free waves, as a first step towards
understanding a nonlinear wave equation $i\partial_t u + \sqrt{-\Delta}\, u=F(u)$, with a suitable non-linearity. The main question here is, given $f_k$ in some function spaces, how does $u$ behave in some other suitable function space? In order to answer this question, one uses the Duhamel formula to represent the solution $u$ as
\begin{equation}\label{representation of the solution}
u(t,x)= \int_{0}^{t}\int_{\R^{nN} } m(\Xi)\, \prod_{j=1}^N \left(\widehat{f}_j(\xi_j)\, e^{ix\cdot\xi_j+is |\xi_j|}\right)\,e^{i(t-s)| \xi_1+\cdots+\xi_N|}  \dd\Xi\, \dd s.
\end{equation}
where 
\m{
\widehat{f}(\xi)=\int_{\R^n} f(x)e^{-i x.\xi}\, \dd x
}
is the Fourier transform of $f$. The inner integral in \eqref{representation of the solution} is precisely of the form of the operators whose boundedness are studied in this paper. This is of course along the lines of the far-reaching method of \textit{space-time resonances} which was introduced by P.~Germain, N.~Masmoudi and J.~Shatah and was explored and applied to nonlinear partial differential equations by them in \cite{germmasshat1}, \cite{germmasshat2}, \cite{germmasshat3} and also by F.~Bernicot and P.  Germain in \cite{BernicotGermain1}, \cite{BernicotGermain2}, \cite{BernicotGermain3}. In our case, we are ignoring the effect of the integral in $s$ which amounts to ignoring the effects of the time resonance.

Motivated by \eqref{representation of the solution}, we study \emph{multilinear Fourier integral operators} (abbreviated multilinear FIOs) of the form
\nm{eq:foidef}{
T^\Phase_\sigma(f_1,\dots,f_N)(x) = \int_{\R^{nN}} \sigma(x,\Xi)\prod_{j=1}^N\left(\widehat{f}_j(\xi_j)e^{ix\cdot\xi_j}\right)\, e^{i\Phase(\Xi)}\, \ddd\Xi,
}
where $\sigma$ is an \emph{amplitude} and
\nm{eq:phaseform}{
\Phase(\Xi) = \phase_0(\xi_1+\dots+\xi_N) + \sum_{j=1}^N\phase_j(\xi_j),
}
is a combination of \emph{phase functions} $\phase_j$ ($j=0,1,\dots,N$). Here the terms amplitude and phase function are defined as follows:\\ 
\begin{defn}\label{defn:amplitudes}
For integers $n,N\geq 1$ and $m \in \R$, the set of \emph{(multilinear) amplitudes}  $S^m(n,N)$ is the set of functions $\sigma \in \mathcal{C}^\infty (\R^n \times \R^{nN})$ that satisfy
\m{
|\partial_{\Xi}^{\alpha}\partial_{x}^{\beta} \sigma(x,\Xi)|\leq C_{\alpha,\beta} \langle \Xi\rangle^{m-|\alpha|},
}
for all multi-indices $\alpha$ and $\beta$. Here and in what follows
\m{
\langle \Xi\rangle = \left(1+\sum_{j=1}^N |\xi_j|^2\right)^{1/2} \quad \mbox{for\, $\Xi = (\xi_1,\dots,\xi_N) \in \R^{nN}$ with $\xi_j \in \R^n$, \,j=\emph{1},\dots, N.}
}
The parameter $m$ is referred to as the \emph{order} or \emph{decay} of the amplitude.
\end{defn}

\begin{defn}\label{defn:linear phase}
A function $\phase\colon \R^n \to \R$ which belongs to $\mathcal{C}^{\infty}(\mathbb{R}^n\setminus \{0\})$ and is positively homogeneous of degree one \emph{(that is satisfies $\varphi (t\xi)= t \varphi(\xi)$ for all $\xi\in \R^n$ and all $t>0$)} is called a \emph{phase function} \emph{(or phase)}.
\end{defn}

In order to state the main result of this paper, i.e. Theorem \ref{thm:main} below, we define
\begin{equation} \label{defn:xp}
    X^p:=\begin{cases}
        h^p & \mbox{if $p\leq 1$}\\
        L^p & \mbox{if $1<p<\infty$}\\
        \bmo & \mbox{if $p=\infty$},
    \end{cases}
\end{equation}
where $L^p$ is the usual Lebesgue space, $h^p$ is the local Hardy space defined in Definition \ref{def:Triebel} below, and $\bmo$ is the dual space of $h^1$. We remind the reader that $L^p$ and $h^p$ coincide when $1 < p < \infty$.

\begin{thm}\label{thm:main}
Given integers $n, N \geq2$ and exponents $p_j \in [1,\infty]$ \emph{($j=0,\dots, N$)} which satisfy
\begin{equation}\label{eq:hoelder}
    \frac{1}{p_0}=\sum_{j=1}^N\frac{1}{p_j},
\end{equation}
suppose that
\begin{equation}\label{ineq:criticalexponent}
	m\leq -(n-1)\brkt{\sum_{j=1}^N\abs{\frac{1}{p_j}-\frac{1}{2}}+\abs{\frac{1}{p_0}-\frac{1}{2}}}.
\end{equation}
Then if $\sigma \in S^{m}(n,N)$ and $\Phi$ is of the form \eqref{eq:phaseform} with each phase $\phase_j$ being as in \emph{Definition \ref{defn:linear phase} ($j=0,1,\dots,N$)}, then the multilinear operator $T^\Phi_\sigma$ initially defined by \eqref{eq:foidef} for $f_1,\dots,f_N \in \mathscr{S}$ $($the Schwartz class$)$, extends to a bounded multilinear operator from $X^{p_1}\times \ldots \times X^{p_N}$ to $X^{p_0}$.
\end{thm}

We can compare this result with the earlier work \cite{Monster} by the same authors. The first novelty of the present result is its global nature in the sense that it doesn't require the amplitudes $\sigma(x,\Xi)$ to be compactly supported in the spatial variable $x$. Indeed, this paper establishes the first global results to date for multilinear (or even bilinear) Fourier integral operators. The second novelty is that we allow a component of the phase function of $T^\Phase_\sigma$ to depend on a mix of the variables $\xi_1, \dots, \xi_N$ in a way that is dictated by the nonlinear wave equation applications, as demonstrated above. In \cite{Monster}, the phase $\phase_0$ was not present (that is, it was identically zero). The third novelty is that the results are proved for multilinear and not just bilinear operators as was the case in \cite{Monster}. There is also a difference in the function spaces considered. In \cite{Monster} the end-point function spaces whose products formed the domain of the operator were the real Hardy space $H^1$ and its dual $\BMO$, whereas here we consider the larger function space $h^1$ and its dual $\bmo$. Although $h^1$ was used as an important technical tool in \cite{Monster}, here it is centre stage. In \cite{Monster} the restriction $p_0 \geq 1$ was not imposed and the target space $X^{p_0}$ was $L^{p_0}$ even for $p_0 \leq 1$. The natural improvement to consider here would be $X^{p_0}$ equal to the local Hardy space $h^{p_0}$ when $p_0 \leq 1$, but this possibility is reserved for a forthcoming paper.

In proving our multilinear boundedness results, it behoved us also to prove the global regularity of linear Fourier integral operators on local Hardy spaces $h^p$ and local spaces of functions of bounded mean oscillations $\bmo$. The local version of this result is stated in the work of M.~Peloso and S.~Secco \cite{PS}, but is not enough for our purposes. Indeed it is not enough even if the amplitude $\sigma(x,\Xi)$ is assumed to have compact $x$-support. This is because the introduction of the mixed phase $\phase_0(\xi_1+\dots+\xi_N)$ leads to the appearance of global Fourier integral operators in the subsequent high frequency decomposition of the operator, so the more complicated phase appears to necessitate the study of global regularity of linear Fourier integral operators. The global linear regularity is proved by a suitable extension of the method of A.~Seeger, C.~Sogge and E.~Stein \cite{SSS} and the globalisation procedure of M.~Ruzhansky and M.~Sugimoto in \cite{RuzhSug}. We prove this regularity for exponents $p > n/(n+1)$, which differs from the range in the local case where $p$ can take any positive value. However, we prove this is the sharp range in the global setting. In the present paper we only make use of this linear result in the case $p=1$, but the full range of exponents will come into play in the forthcoming paper mentioned above.

Beyond the need to understand global Fourier integral operators, the presence of the mixed phase $\phase_0(\xi_1+\dots+\xi_N)$ leads to other difficulties. The underlying cause of these difficulties is the failure of commutator techniques which were an essential ingredient in \cite{Monster}. To successfully apply such techniques in this context would require better control of the commutator between a linear Fourier integral operator and a multiplication operator (that which is denoted $M_{\mathfrak m}$ in Section \ref{sec:sigmaj}) than seems reasonable to expect. Instead we succeed in decomposing the operators into a sum of the constant coefficient operators (that is, the case $\sigma(x,\Xi)$ does not depend on $x$, which corresponds to $M_{\mathfrak m}$ being the identity operator). This requires at times careful control of the Carleson measure generated by a $\bmo$ function.

The  multilinear results of this paper are then achieved through the following steps. First we identify the end-points that are needed for the complex interpolation which leads to the regularity of multilinear Fourier integral operators on products of $L^p$ spaces. Thereafter we make a multilinear phase space analysis to divide the operator according to various frequency supports of the amplitude.  This creates a number of cases with their associated difficulties, that will be dealt with in accordance to the form of the endpoints in question. Finally complex interpolation yields the main result.

The paper is organised as follows, In Section \ref{subsection:definitions} we recall some definitions, and results from linear and multilinear harmonic and microlocal analysis. In Section \ref{global hp and bmo results} we prove the global $h^p$ and $\bmo$ regularity of Fourier integral operators using among other things, Ruzhansky-Sugimoto's globalisation procedure and in Section \ref{Sharpness} we show that the results are actually sharp. Section \ref{endpoint cases} is devoted to finding the so-called endpoints for which the complex interpolation would provide the final regularity result for multilinear Fourier integral operators. Finally in Sections~\ref{sec:sigmaj}, \ref{sec:sigmaij} and \ref{sec:sigma0} we systematically analyse all the endpoint cases for various frequency localisations.

\section{Definitions and Preliminaries}\label{subsection:definitions}
In this section, we will collect all the definitions that will be used throughout this paper. We also state some useful results from both harmonic and microlocal analysis which will be used in the proofs of our results.

The proof of Theorem \ref{thm:main} builds upon the corresponding linear results. Indeed, as mentioned in Section \ref{sec:intro}, the proof we present requires new linear boundedness results. We begin by recalling the definitions of the linear versions of the main objects of study in this paper. The multilinear amplitudes defined in Definition \ref{defn:amplitudes} reduce to the classical H\"ormander classes $S^m$ of \emph{amplitudes} (or \emph{symbols}) in the case $N=1$, that is to say $S^m = S^m(n,1)$. The same is true of linear Fourier integral operators: They are the special case of \eqref{eq:foidef} when $N=1$, so in that case we write
\begin{equation*}
T_a^{\phase}f(x) := \int_{\R^n}e^{ix\cdot\xi +i\varphi(\xi)}a(x,\xi)\widehat{f}(\xi) \ddd \xi,
\end{equation*}
for a given amplitude $a \in S^m$ and phase function $\phase$. Such an operator is called pseudodifferential operator under the further restriction that $\phase \equiv 0$. In this case it is useful to introduce slightly different (although widely used) notation: For $a \in S^m$ we define a (\emph{linear}) \emph{pseudodifferential operator} to be the operator
\begin{equation*}
a(x,D)f(x) := \int_{\R^n}e^{ix\cdot\xi}a(x,\xi)\widehat{f}(\xi) \ddd \xi,
\end{equation*}
which, as is the case for all FIOs, is a priori defined on the Schwartz class $\mathscr{S}(\R^n).$ The terminology \emph{symbol} is typically used in connection with pseudodifferential operators and \emph{amplitude} in connection with Fourier integral operators.

We will denote constants which can be determined by known parameters in a given situation, but whose values are not crucial to the problem at hand, by $C$ or $c$, sometimes adding a subscript, for example $c_\alpha$, to emphasis a dependency on a given parameter $\alpha$. Such parameters are those which determine function spaces, such as $p$ or $m$ for example, the dimension $n$ of the underlying Euclidean space, and the constants connected to the seminorms of various amplitudes or phase functions. The value of the constants may differ from line to line, but in each instance could be estimated if necessary. We also write $a\lesssim b$ as shorthand for $a\leq Cb$ and $a\approx b$ when $a\lesssim b$ and $b\lesssim a$. By
\m{
B(x,r) := \{ y\in\R^n \, : \, |y-x| < r\}
}
we denote the open ball of radius $r > 0$ centred at $x\in\R^n$.

The following partition of unity is a standard tool in harmonic analysis and is even used to define the function spaces that we are concerned with.

\begin{defn}\label{def:LP}
Let $\psi_0 \in \mathcal C_c^\infty(\R^n)$ be equal to $1$ on $B(0,1)$ and be supported in $B(0,2)$. We define
$$\psi_j(\xi) := \psi_0 \left (2^{-j}\xi \right )-\psi_0(2^{-(j-1)}\xi),$$
for integers $j\geq 1$. Then one has the following \emph{Littlewood-Paley partition of unity}:
\nm{eq:littlewoodpaley}{
\sum_{j=0}^\infty \psi_j(\xi) = 1 \quad \text{\emph{for all }}\xi\in\R^n.
}
\end{defn}
With the help of the Littlewood-Paley partition of unity we define local Hardy spaces first introduced by D.~Goldberg~\cite{Gol}.

\begin{defn}\label{def:Triebel}
For each $0 < p <\infty$ the following characterisations of the \emph{local Hardy space} $h^p(\R^n)$ are equivalent. See, for example, \cites{Triebel} and \cites{PS}.
\begin{enumerate}[label=\emph{(}\roman*\emph{)}, ref=(\roman*)]
\item

The set of tempered distributions $f \in {\mathscr{S}}'(\R^n)$ such that
\begin{equation*}\label{eq:hp basic}
	\|f\|_{{h}^p(\R^n)}^{[1]} := \brkt{\int \sup_{0<t<1} \abs{\psi_0(tD) f(x)}^p\dd x}^{\frac 1 p}<\infty.
\end{equation*}

\item \label{defn:hpsquarefunction} The set of all $f \in {\mathscr{S}}'(\R^n)$ such that
    \begin{equation*}
    \|f\|_{{h}^p(\R^n)}^{[2]} := \norm{\left(	\sum_{j=0}^\infty|\psi_j(D)f|^{2} \right)^{\frac{1}{2}}}_{L^p(\R^n)} <\infty.
    \end{equation*}

\item \label{defn:hpatomic} Fix
    \m{
    M\geq \left\lfloor n\brkt{\frac 1p -1}_+\right\rfloor,
    }
    where $\lfloor x \rfloor$ denotes the integer part of $x$. The set of all $f \in {\mathscr{S}}'(\R^n)$ for which there exist a sequence $(\lambda_j)_{j=1}^\infty$ of numbers and a sequence $(a_j)_{j=1}^\infty$ of $(h^p,M)$-atoms \emph{(abbreviated $h^p$-atoms below)} such that
    \m{
    f=\sum_{j}\lambda_{j}a_{j}
    }
    and
    \begin{equation*}
    \|f\|_{{h}^p(\R^n)}^{[3]} := \left(\sum_j |\lambda_j|^p\right)^{1/p} < \infty.
    \end{equation*}
    A function $a$ is called an $(h^p,M)$-atom if for some $x_0\in \R^n$ and $r>0$ the following three conditions are satisfied:
\begin{enumerate}[label=\emph{(\alph*)}, ref=(\alph*)]
\item $\supp a\subseteq B(x_{0}, r)$;
\item $\displaystyle |a(x)|\leq|B(x_{0}, r)|^{-\frac 1p}$; and
\item If $r\leq 1$ and $|\alpha|\leq M$, then
\m{\int_{\R^n} x^{\alpha}a(x)\dd x=0.} 
\end{enumerate}
   \item \label{defn:hppelososecco} The set of all $f \in {\mathscr{S}}'(\R^n)$ such that
    \begin{equation*}
     \Vert f\Vert _{h^p}^{[4]} := \Vert \psi_0(D) f \Vert_{L^p}+\sum_{M\leq |\alpha|\leq M+1} \sup_{0<\varepsilon\leq 1}\Vert r^{\alpha}_{\varepsilon} (D) f\Vert_{L^p} < \infty,
     \end{equation*}
     where $M$ is as in the characterisation $(ii)$ above,  and
    \m{
    r_{\varepsilon}^{\alpha} (\xi)= \psi_0(\varepsilon \xi) \prod_{i=1}^n \big(\frac{\xi_i}{|\xi|}\Big)^{\alpha_i} (1-\psi_0(\xi))^{\alpha_i}.
    }
    \item \label{eq:hp} The set of all $f \in {\mathscr{S}}'(\R^n)$ such that
\begin{equation*}
	\|f\|_{{h}^p(\R^n)}^{[5]} := \brkt{\int \sup_{0<t<1}\sup_{\abs{x-y}<t} \abs{\psi_0(tD) f(y)}^p\dd x}^{\frac{1}{p}} < \infty.
\end{equation*}
\end{enumerate}
Moreover all the norms here are equivalent, that is
\m{
\|f\|_{{h}^p(\R^n)}^{[1]} \thickapprox \|f\|_{{h}^p(\R^n)}^{[2]} \thickapprox \|f\|_{{h}^p(\R^n)}^{[3]} \thickapprox \|f\|_{{h}^p(\R^n)}^{[4]} \thickapprox \|f\|_{{h}^p(\R^n)}^{[5]},
}
with implicit constants that only depend on the dimension $n$ and the choice of $\psi_0$ in the Littlewood-Paley decomposition, so we simply write $\|f\|_{{h}^p(\R^n)}$ for all of them.
\end{defn}
{{It is also shown in \cite{Gol} that a function $f$ belongs to the local Hardy space $h^1$ if, and only if $f\in L^1$ and $\mathfrak{R}_j((1-\psi_0))(D)f)\in L^1$ where $\mathfrak{R}_j$ denotes the $j$-th Riesz transform, i.e. 
$\widehat{\mathfrak{R}_{j}f}(\xi):= -i\frac{\xi_j}{|\xi|}\widehat{f}(\xi),$ $j=1,\dots, n.$ We record here for future use the more familiar special case of Definition \ref{def:Triebel} \ref{defn:hppelososecco} when $p=1$:
\begin{equation}\label{eq:local_H1}
\begin{split}
	\norm{f}_{h^1} &\thickapprox \norm{f}_{L^1}+\sum_{j=1}^n \norm{\mathfrak{R}_j((1-\psi_0)(D)f)}_{L^1}\thickapprox \norm{\psi_0 (D)f}_{L^1} +\norm{(1-\psi_0)(D) f}_{H^1}.
\end{split}
\end{equation}}}
The dual of the local Hardy space $h^1$ is the \textit{local} $\BMO$ space, which is denoted by $\bmo$ and consists of locally integrable functions that verify
\begin{equation}\label{defn:bmo}
\Vert f\Vert_{\bmo}:= \Vert f\Vert_{\BMO}+ \Vert \psi_0 (D) f\Vert_{L^{\infty}}<\infty,
\end{equation}
where $\BMO$ is the usual John-Nirenberg space of functions of bounded mean oscillation (see \cite{S} for the definition) and $\psi_0$ is the cut-off function introduced in Definition \ref{def:LP}.

To bound the low frequency part of an FIO, where the phase function is singular, we will make use of the following lemma, whose proof is a scholium of Lemma 1.17 in \cite{DW}, and therefore left to the reader.\\ 

\begin{lem}\label{main low frequency estim}
     Let $a(\xi)\in C^{\infty}_{c}(\R^n)$ be supported in a neighbourhood of the origin. Assume also that $\varphi (\xi)\in C^{\infty}(\R^n \setminus 0)$, is positively homogeneous of degree one.
 Then for all $0 \leq \varepsilon <1$ we have
     \begin{align*}
     \label{LowFreq:KernelEst1}
          \Big| \int e^{i\varphi(\xi)-i x\cdot\xi } a(\xi) \, \dd \xi \Big| \lesssim \langle x\rangle^{-n-\varepsilon}.
     \end{align*}
\end{lem}
The following lemma will also prove useful in bounding the low frequency part of an FIO. It is a consequence of a result due to J.~Peetre \cite{Peetre}.
\begin{lem}\label{peetre lemma}
Let $f\in \mathcal C^1(\R^n)$ have Fourier support contained inside the unit ball. Then for every $\rho>n$, and $r \in ( n/\rho,1]$ one has
\begin{equation}
    \left (\langle \cdot\rangle^{-\rho} \ast |f|\right )(x)\lesssim \Big (M(|f|^r)(x)\Big ) ^{1/r}, \quad x \in \mathbb{R}^n,
\end{equation}
where $M$ denotes the Hardy-Littlewood maximal function on $\R^n$. 

\end{lem}
\begin{proof}
As was shown by Peetre, see e.g. \cite{Triebel}*{Section 2.3.6}, one has for 
$r\geq n/\rho$ that 
\begin{equation}\label{Peetres inequality}
    \sup_{y\in\R^n} \frac{|f(x-y)|}{\langle y\rangle^{\rho}}
\lesssim \Big( M(|f|^r)(x)\Big)^{1/r}.
\end{equation}
Now taking $r\in (n/\rho, 1]$, and using \eqref {Peetres inequality} we obtain
\begin{align*}
|\langle \cdot\rangle^{-\rho} \ast f(x)|
& \lesssim \int_{\R^n} \frac{|f(x-y)|}{\langle y\rangle^{\rho} } \dd y
\leq \Big(\sup_{y\in\R^n} \frac{|f(x-y)|}{\langle y\rangle^{\rho} }\Big)^{1-r} 
\int_{\R^n} \frac{|f(x-y)|^r}{\langle y\rangle^{\rho r} } \dd y
\\ & \lesssim \Big(M(|f|^r)(x)\Big)^{1/r}.\qedhere
\end{align*}
\end{proof}



In the analysis of multilinear operators, a basic tool is a certain type of measure whose definition we now recall. A Borel measure $\dddd\mu(x,t)$ on $\R^{n+1}_+$ is called a \emph{Carleson measure} if
\m{
\Vert \dddd \mu\Vert_{\mathcal{C}} := \sup_Q \frac{1}{|Q|}\int_0^{\ell(Q)}\int_Q |\dddd\mu(x,t)| < \infty
}
where the supremum is taken over cubes all $Q \subset \R^n$ and $\ell(Q)$ denotes the diameter of $Q$ and $|Q|$ its Lebesgue measure. The quantity $\Vert \dddd \mu\Vert_{\mathcal{C}}$ is called the \emph{Carleson norm} of $\dddd \mu$. In this paper we are exclusively interested in Carleson measures which are supported on lines parallel to the boundary of $\R^{n+1}_+$. More precisely, in what follows all Carleson measures will be supported on the set
\m{
E := \{(x,t) \, : \, \mbox{$x\in\R^n$ and $t=2^{-k}$ for some $k\in\Z$}\}
}
so they take the form
\m{
\sum_{k\in\Z}\dddd\mu(x,t)\delta_{2^{-k}}(t),
}
where $\delta_{2^{-k}}(t)$ is a Dirac measure at $2^{-k}$. This will be assumed throughout without further comment.

We recall some basic results concerning Carleson measures due to L. Carleson \cite{Carl} which are also (as we shall see) useful in the context of multilinear operators. See also E. M. Stein \cite{S} for more streamlined and simplified proofs of the following results.
{{
\begin{lemma}\label{Coifman-Meyerslem}
If $\dd \mu(x, t)$ is a Carleson measure, then 
\begin{equation*}
    \sum_{k}\int _{\R^n}F(x,2^{-k}) \, \dd \mu(x, 2^{-k})\leq C_n \norm{ \dd \mu}_{\mathcal{C}}\int (\sup_{k} \sup_{\abs{y-x} < 2^{-k}} |F(y,2^{-k})|) \dd x.
\end{equation*}
moreover for $0<p<\infty$ one has 
\begin{equation}\label{carleson estim}
    \sum_{k}\int _{\R^n}|F(x,2^{-k})|^p \, \dd \mu(x, 2^{-k})\leq C_n \norm{ \dd \mu}_{\mathcal{C}}\int (\sup_{k} \sup_{\abs{y-x} < 2^{-k}} |F(y,2^{-k})|)^p \dd x.
\end{equation}

Consequently, if $\varphi$ satisfies $|\varphi(x)|\lesssim \langle x\rangle^{-n-\varepsilon}$ \emph{(for some $0<\varepsilon<\infty$), then}
\begin{equation}\label{ineq:carll2}
    \sum_{k}\int _{\R^n}|\varphi(2^{-k}D) f(x)|^2 \, \dd \mu(x, 2^{-k})\leq C_n \norm{\dddd \mu}_{\mathcal{C}} \norm{f}_{L^2}^2,
\end{equation}
and if $\varphi$ is a bump function supported in a ball near the origin with $\phi(0)=1$ then one also has
\begin{equation}\label{ineq:carlh1}
    \sum_{k}\int _{\R^n}|\varphi(2^{-k}D) f(x)| \, \dd \mu(x, 2^{-k})\leq C_n \norm{\dd \mu}_{\mathcal{C}} \norm{f}_{h^1}.
\end{equation}
\end{lemma}}}
We also recall the quadratic estimate which is a consequence of Plancherel's Theorem: If $\varphi \in {\mathscr{S}}$ is such that $\varphi(0) = 0$, then
\nm{ineq:quadraticestimate}{
\sum_k \int \abs{\varphi(2^{-k}D)f(x)}^2 \dd x \lesssim \norm{f}_{L^2}^2.
}
Finally, we shall also use the following result which was stated and proved as Lemma~4.10 in \cite{Monster}
\begin{lemma}\label{lem:carlesoncomp} 
For any Carleson measure $\dddd\mu$ supported on $E$ and $K_k$ satisfying 
\begin{equation*}
	\abs{K_k(x-y)}\lesssim {2^{kn}}{\brkt{1+\frac{\abs{x-y}}{2^{-k}}}^{-n-\delta}}
\end{equation*}
for some $\delta>0$,
one has that
\[
	\dd \tilde{\mu}(x,t):=\sum_k\brkt{\int \abs{K_k(x-y)}\dd \mu(y,t)}\delta_{2^{-k}}(t)\dd x,
\]
defines a Carleson measure and $\norm{\dddd \tilde{\mu}}_{\mathcal{C}}\lesssim \norm{\dddd\mu}_{\mathcal{C}}$.
\end{lemma}

As stated in Section \ref{sec:intro} Theorem \ref{thm:main} is proved by interpolating between certain end-point cases.  In connection to those end-point cases, the Hardy space $H^1$ and its dual $\mathrm{BMO}$ (see \cite{S} for the definitions) will play an important role. In this context the following variant of Corollary 4.12 in \cite{Monster} will be useful.
\begin{prop}
\label{cor:monster 1} Let ${\psi} \in \mathscr S(\R^n)$ be supported in an annulus and ${\phi} \in\mathscr S(\R^n)$ satisfy ${\phi}(0) = 0$. Then for any $F\in H^1$, $G\in \mathrm{BMO}$ and $v\in L^\infty_{k,x}$,
\[
	\Big|\int \sum_{k = -\infty}^\infty {\psi}(2^{-k}D) F(x)\, {\phi}(2^{-k}D) G(x)\, v(2^{-k},x)\,\dd x\Big|
	\lesssim \norm{F}_{H^1}\norm{G}_{\mathrm{BMO}}\norm{v}_{L^\infty_{k,x}}.
\]
\end{prop}

\section{Global $h^p\to h^p$ boundedness of linear FIO's for $n/(n+1)<p<1$}\label{global hp and bmo results}

In this section we establish the global $h^p$ boundedness of a class of linear FIOs. This is formulated as Theorem \ref{linearhpthm} below and will be needed to prove Theorem \ref{thm:main}. Since $H^1 \subset h^1 \subset L^1$, this result strengthens the global $H^1$ to $L^1$ boundedness obtained by Ruzhanski and Sugimoto~\cite{RuzhSug} for these FIOs. It also extends the local $h^p\to h^p$ boundedness of FIOs proven by Peloso and Secco~\cite{PS} to a global result, that is to say, we remove the requirement that the amplitude have compact $x$-support.

While this article was being written this result was generalised further to cover more general phases and Besov-Lipschitz, as well as Triebel-Lizorkin spaces. This generalisation is presented in detail in the paper of the first and third authors together with A.~Israelsson~\cite{IRS}. Therefore we concentrate here on presenting the main ideas of this result and skip some of the technical details. The interested reader can find these details in \cite{IRS}.

\begin{thm}\label{linearhpthm}
Let $m=-(n-1)\abs{\frac{1}{p}-\frac{1}{2}}$ and $\frac{n}{n+1}<p\leq\infty$. Then any linear Fourier integral operator
$$T_\sigma^{\phase}f(x)= \int_{\R^n} \sigma(x,\xi)\, e^{ix\cdot \xi +i\phase(\xi)} \widehat{f}(\xi) \ddd\xi ,$$
with an amplitude $\sigma(x,\xi)\in S^{m}_{1,0}$ and a phase function $\varphi$ \emph{(as in Definitions \ref{defn:amplitudes} and \ref{defn:linear phase})}, satisfies the estimate
\[
	\norm{T_\sigma^{\phase} f}_{X^p}\leq C\norm{f}_{X^p},
\]
where $X^p$ is defined in \eqref{defn:xp}.

\end{thm}

We begin the proof of Theorem \ref{linearhpthm} by reducing to the case of $x$-independent amplitudes and $p < \infty$. We can write
\m{
T_\sigma^{\phase}f(x)= b(x,D)T^{\varphi}_{\widetilde{\sigma}} f(x),
}
where $b(x,\xi) = \sigma(x,\xi)\langle \xi\rangle^{-m}\in S^{0}$ and $\widetilde{\sigma} = \langle \xi\rangle^m \in S^m$ is independent of $x$. Since pseudodifferential operators $b(x,D)$ are bounded on $X^p$ (see \cite{Gol} for the case $p\leq1$ and $p=\infty$, and, for example, \cites{S} for $1<p<\infty$) to prove Theorem \ref{linearhpthm} we only need to prove the boundedness of $T^{\varphi}_{\widetilde{\sigma}}$. Since $T^{\varphi}_{\widetilde{\sigma}}$ is a self-adjoint operator, duality implies that the $p=\infty$ case follows immediately from the $p=1$ case. To avoid unnecessarily cumbersome notation, for the rest of the proof we drop the tilde and assume $\sigma$ only depends on $\xi$.

Next we observe that the $L^2$ boundedness of $T^\phase_\sigma$ is obvious when $\sigma$ does not depend on $x$, since it is a Fourier multiplier with bounded symbol (observe that $m\leq 0$). Therefore, we only need to consider $p \in (n/(n+1),1]$, since once the theorem is proved for these values of $p$, the others follow by interpolation and duality. 

We now split the operator into high and low frequency portions. Let $\chi(\xi)$ be a smooth cut-off function supported in the ball $B(0,1)$ and equal to one in $B(0,1/2)$. We set
\m{
\sigma_1 := \chi(\xi)\sigma(\xi), \quad \mbox{and} \quad \sigma_2(\xi):= (1-\chi(\xi))\sigma(\xi),
}
so $\sigma = \sigma_1 + \sigma_2$. We shall study the boundedness of $T_{\sigma_1}^{\phase}$ and $T_{\sigma_2}^{\phase}$ separately
and begin with the estimates for $T_{\sigma_1}^\phase$.
\subsection{Low frequency analysis}\label{sec:lowfreqlinearhp}

Our goal is to show that $T_{\sigma_1}^{\varphi}$ is bounded on $h^p$ for $\frac{n}{n+1}<p<\infty$. For this we make use of the characterisation \ref{defn:hpsquarefunction} in Definition \ref{def:Triebel} and let $\psi_{j}$ be a standard Littlewood-Paley partition of unity introduced in Definition \ref{def:LP}.

Clearly the operator $\psi_j(D)T^{\varphi}_{\sigma_1}$ is an FIO with amplitude $$r_j(\xi)= \psi_{j}(\xi)\sigma_1(\xi) = \psi_{j}(\xi)\chi(\xi) \sigma(\xi)$$ and phase function $x\cdot\xi +\varphi(\xi)$. The support properties of $\psi_j$ and $\chi$ imply that $r_j(\xi)=0$ for $j\geq 1$. This yields that
\m{
\norm{T_{\sigma_1}^{\phase}f}_{h^p} = \norm{\left(	\sum_{j=0}^\infty|T_{r_j}^{\phase}f|^{2} \right)^{\frac{1}{2}}}_{L^p} = \norm{T_{r_0}^{\phase}f}_{L^p}.
}
We can write
$$T_{r_0}^{\phase}f(x)= \int K(x,y) (\psi_0 (D) f)(y)\, \dd y, $$
where $K(x,y)=\int \chi (\xi)\, \sigma (\xi)\, e^{i(x-y)\cdot\xi +i\varphi(\xi)}\, \ddd \xi$. By Lemma \ref{main low frequency estim}, one has that $|K(x,y)|\lesssim \langle x-y \rangle^{-n-\varepsilon}$ for all $\varepsilon \in [0,1)$. Using this and Lemma \ref{peetre lemma} yields that
 \begin{equation}\label{poitwise estimate for the lowfrequency part}
 |T_{r_0}^{\phase}f(x)|\lesssim |(\psi_0 (D) f) \ast \langle \cdot \rangle^{-n-\varepsilon}|\lesssim \big( M(|\psi_0 (D) f|^r) (x)\big)^{1/r}
 \end{equation}
 for all $f\in \mathscr{S}$, $r\in (\frac{n}{n+\varepsilon},1)$ and $\varepsilon \in (0,1)$, where $M$ is the Hardy-Littlewood maximal function.
 
Thus, by choosing $\frac{n}{n+1}<r<p$ and making use of the boundedness of $M$ on $L^{p/r}$ we obtain
\m{
\norm{T_{\sigma_1}^{\phase}f}_{h^p}=\Vert T^{\varphi}_{r_0} f\Vert_{L^p} \lesssim \Vert M(|\psi_0 (D) f|^r)\Vert_{L^{p/r}}^{1/r} \lesssim \Vert \psi_0 (D) f\Vert_{L^{p}} \lesssim \Vert f\Vert_{h^p},
}
where the last inequality follows by \ref{defn:hpsquarefunction} in Definition \ref{def:Triebel}. A standard density argument yields the result.

\subsection{High frequency analysis}

To analyse $T^\phase_{\sigma_2}$ we need to use the atomic characterisation/decomposition of local Hardy spaces, that is \ref{defn:hpatomic} of Definition \ref{def:Triebel}. It is also worth mentioning that the high frequency case of the proof does not require the restriction $p \in ({n}/{n+1}, 1)$ and works for all $p\in(0,1)$. Indeed, it is the lack of smoothness in the low frequency part of the operator that leads to the counter-example in Section \ref{Sharpness}.

We can make a further reduction and replace the target space $h^p$ with the larger space $L^p$ by using the characterisation \ref{defn:hppelososecco} from Definition \ref{def:Triebel}. This characterisation states that it is enough to show that $r_{\varepsilon}^{\alpha}(D)\circ T^{\varphi}_{\sigma_2}$ and $\psi_0(D)\circ T^{\varphi}_{\sigma_2}$ both map $h^p$ to $L^p$, with the norm of the former uniform in $\varepsilon$. But this follows at once from the facts that $r_{\varepsilon}^{\alpha}(D)$ and $\psi_0(D)$ are pseudodifferential operators with symbols in $S^0$ (uniformly in $\varepsilon$) and $\cap_{\mu \leq 0} S^{\mu}$ respectively, and $r_{\varepsilon}^{\alpha}(D)T^{\varphi}_{\sigma_2}= T^{\varphi}_{\sigma_2}r_{\varepsilon}^{\alpha}(D) $ and $\psi_0 (D)T^{\varphi}_{\sigma_2}= T^{\varphi}_{\sigma_2}\psi_0 (D).$

\subsubsection{Estimates of the norm on small balls}\label{sec:smallballs}

We introduce a second frequency decomposition to the Littlewood-Paley decomposition of Definition \ref{def:LP}. This was inspired by the work of C.~Fefferman~\cite{Feffer} and famously used by A.~Seeger, C.~Sogge and E.~Stein in \cite{SSS}. This section follows closely the same line of thought as \cite{SSS}, in which each Littlewood-Paley shell $\{\xi \colon 2^{j-1}\leq \vert \xi\vert\leq 2^{j+1}\}$ is further partitioned into $O(2^{j(n-1)/{2}})$ truncated cones of thickness $2^{j/2}$, and a clear exposition of the claims made below can be found in \cite{S}*{pp.~402--12}.

For each $j\in\N$ we choose a collection of unit vectors $\{\xi^{\nu}_{j}\}_\nu$ such that
\begin{itemize}
\item $\displaystyle \big | \xi^{\nu}_{j}-\xi^{\nu'}_{j} \big |\geq 2^{-\frac{j}{2}}$ for $\nu\neq \nu'$, and
\item for each  $\xi\in\mathbb{S}^{n-1}$, there exists a $\xi^{\nu}_{j}$ such that $\big \vert \xi -\xi^{\nu}_{j} \big \vert <2^{-{j}/{2}}$,
\end{itemize}
which is maximal with respect to the first property. It follows that it contains at most $O(2^{j(n-1)/2})$ elements. Associated to each $\xi_j^\nu$ is a cone
\begin{equation*}\label{eq:gammajnu}
    \Gamma^{\nu}_{j}:=\set{ \xi\in\R^n :\,  \abs{ \frac{\xi}{\vert\xi\vert}-\xi^{\nu}_{j}}\leq 2\cdot 2^{-\frac{j}{2}}}.
\end{equation*}
whose central axis lies along $\xi_j^\nu$.

One can construct a partition of unity
\begin{equation}\label{the first partition of unity}
\sum _{\nu}\chi^{\nu}_{j} =1
\end{equation}
of $\R^n\setminus\set{0}$ subordinate to $\{\Gamma^{\nu}_{j}\}_{j,\nu}$ which satisfies the estimates
\nm{eq:quadrseconddyadcond}{
\abs{\d_\xi^\alpha \chi_j^\nu (\xi) }\leq C_{\alpha} 2^{j\frac {\abs \alpha}2}\abs\xi^{-\abs\alpha}
}
for all multi-indices $\alpha$, and the better estimate
\nm{eq:quadrseconddyadcond2}{
\left \vert (\xi_j^\nu\cdot \nabla)^{N}\chi^{\nu}_{j}(\xi) \right \vert\leq C_{N}
\vert \xi\vert ^{-N},
}
for $N\geq 1$ along the direction $\xi_j^\nu$. Therefore, with $\psi_j$ from Definition \ref{def:LP},
\begin{equation}\label{PL partition of unity}
\psi_{0}(\xi)+\sum_{j=1}^{\infty}\sum_{\nu}\chi_{j}^{\nu}(\xi)\psi_{j}(\xi)=1, \quad \text{for all } \xi\in \mathbb{R}^{n}.
\end{equation}

We now fix an $h^p$-atom $a$ supported in a ball $B(\overline{y}, r)$ with $r\leq 1$. We need to show that $\Vert Ta\Vert_{L^p} \leq C$, where the constant $C$ does not depend on the atom $a$ or the radius of its support $r$. To do this we introduce the rectangles
$$
R_{j}^{\nu}=\set{x\in\ \R^n\ :\ |x-\bar y+\nabla_{\xi}\varphi(\xi_{j}^{\nu})|\leq A 2^{-\frac j2},\ |\pi_{j}^{\nu}(x-\bar y+\nabla_{\xi}\varphi(\xi_{j}^{\nu}))|\leq A 2^{-j}},
$$
where $\pi_j^\nu$ is the orthogonal projection in the direction $\xi_{j}^{\nu}$, and the size of the constant $A$ depends on the size of the Hessian $\partial^2_{\xi\xi}\varphi$ but not on $j$, and define the ``region of influence" as
\begin{equation*}\label{defn:Bstar}
B^{*}=\bigcup_{2^{-j}\leq r}\bigcup_{\nu}R_{j}^{\nu}.
\end{equation*}
We then split
\begin{equation}\label{eq: spliting the FIO on balls}
  \int_{\mathbb{R}^{n}}|T^\phase_{\sigma_2} a(x)|^{p}\mathrm{d}x=\int_{B^{*}}|T^\phase_{\sigma_2} a(x)|^{p}\mathrm{d}x+\int_{B^{*c}}|T^\phase_{\sigma_2}a(x)|^{p}\mathrm{d}x,
\end{equation}
It can be shown that
\m{
|B^{*}|\lesssim r,
}
so
\begin{equation} \label{eq:onsetofinfluence}
\int_{B^{*}}|T^\phase_{\sigma_2} a(x)|^{p}\mathrm{d}x\leq|B^{*}|^{1-p/2}\brkt{\int_{B^{*}}|T^\phase_{\sigma_2} a(x)|^{2}\mathrm{d}x}^{p/2}
\lesssim r^{(1-p/2)}\Vert T^\phase_{\sigma_2} a\Vert_{L^{2}}^{p}.
\end{equation}
To estimate $\norm{T^\phase_{\sigma_2} a}_{L^{2}}^{p}$ we consider two cases: $-{n}/{2}<m\leq0$; and $m\leq {-n}/{2}$.

In the case $-{n}/{2}<m\leq0$ we can fix $q \in (1,2]$ which satisfies
\nm{eq:rieszpotentialexponents}{
\frac{1}{2}=\frac{1}{q}+\frac{m}{n}.
}
Using the $L^2$-boundedness of $T^\phase_{\sigma_2}\circ\langle D \rangle^{-m}$ (which is clear when viewed as a zeroth-order Fourier multiplier) and the $L^q$ to $L^2$ boundedness of the Riesz potential $\langle D \rangle^{m}$, we obtain
$$
\Vert T^\phase_{\sigma_2}a\Vert_{L^{2}}^{p}\lesssim\Vert a\Vert_{L^{q}}^{p}\lesssim c|B|^{p/q -1}\lesssim r^{n(p/q -1)}.
$$
Combining this with \eqref{eq:onsetofinfluence} we obtain
$$
\int_{B^{*}}| T^\phase_{\sigma_2}a(x)|^{p}\mathrm{d}x\lesssim r^{(1-p/2)+n(p/q -1)} \lesssim 1,
$$
where the last estimate follows from \eqref{eq:rieszpotentialexponents} since then
\ma{
1-\frac{p}{2}+n\Big(\frac{p}{q} -1\Big) &= 1-\frac{p}{2}+n\Big(\frac{p}{2}-\frac{pm}{n}-1\Big)= p\Big(\frac{1}{p}-\frac{1}{2}+\frac{n}{2}-m-\frac{n}{p}\Big)\\
&= p\Big[-(n-1)\Big(\frac{1}{p}-\frac{1}{2}\Big)-m \Big] = 0.
}

If instead $m\leq \frac{-n}{2},$ then by setting $b=|B|^{1/p-1/q}a$, with $q$ once again satisfying \eqref{eq:rieszpotentialexponents} (so now $q < p < 1$) we see that $b$ is a $h^q$-atom which is also supported in $B$. In fact, since $r \leq 1$, $b$ is even an atom in $H^{q}$, so by Corollary 2.3 in \cite{Krantz}, we have that $T^\phase_{\sigma_2}$ is bounded from  $H^{q}$ to $L^{2}$, and so
\begin{equation*}
\int_{B^{*}}|T^\phase_{\sigma_2} a(x)|^{p}\mathrm{d}x\lesssim r^{(1-p/2)}\Vert a\Vert_{H^{q}}^{p}
\lesssim r^{(1-p/2)}|B|^{(1/q-1/p)p}\Vert b \Vert_{H^{q}}^{p}
\lesssim r^{(1-p/2)+n(p/q-1)} \lesssim 1
\end{equation*}
once again.

To analyse the second term on the right-hand side of \eqref{eq: spliting the FIO on balls} we use the partition of unity \eqref{PL partition of unity} and decompose
\begin{equation*}\label{defn of sss pieces of T}
 T_{\sigma_{2}}^{\varphi}=\sum_{j=0}^{\infty}T_{j}=\sum_{j}\sum_{\nu}T_{j}^{\nu},
\end{equation*}
where $T_{j}^{\nu}$ is the operator with kernel
\m{
K_j^\nu(x,y) = \int \sigma(\xi)\chi_{j}^{\nu}(\xi)\psi_{j}(\xi)e^{i(x-y)\cdot\xi + i\phase(\xi)} \ddd\xi.
}
Since $\phase$ is homogeneous of degree one we can write $\phase(\xi) = \nabla\phase(\xi)\cdot\xi$ and so
\m{
(x-y)\cdot\xi + \phase(\xi) = \brkt{x-y+\nabla\phase(\xi_j^\nu)}\cdot\xi + \brkt{\nabla\phase(\xi)-\nabla\phase(\xi_j^\nu)}\cdot\xi.
}
Just as in \cite{SSS}, the kernel can therefore be written as
\m{
K_j^\nu(x,y) = \int b_{j}^\nu(\xi)e^{i\brkt{x-y+\nabla\phase(\xi_j^\nu)}\cdot\xi} \ddd\xi.
}
where $b_j^\nu(\xi) := \sigma(\xi)\chi_{j}^{\nu}(\xi)\psi_{j}(\xi)e^{i\brkt{\nabla\phase(\xi)-\nabla\phase(\xi_j^\nu)}\cdot\xi}$ satisfies the estimates
\m{
\abs{\d^\alpha b_j^\nu (\xi) }\leq C_{\alpha} 2^{-j\frac {\abs \alpha}2} \quad \mbox{and} \quad \left \vert (\xi_j^\nu\cdot \nabla)^{\NNN}b^{\nu}_{j}(\xi) \right \vert\leq C_{\NNN}
2^{-j\NNN/2},
}
for all multi-indices $\alpha$ and for $\NNN\geq 1$, in a similar way to \eqref{eq:quadrseconddyadcond} and \eqref{eq:quadrseconddyadcond2}. This leads to the kernel estimate
\nm{eq:claimkernel}{
|{\partial_y^\alpha}K_{j}^{\nu}(x, y)|  \lesssim \frac{2^{j\brkt{m+\frac{n+1}{2}+\abs\alpha}}}{\left  (1+\big|2^{j}\pi_j^\nu(x- y+\nabla_{\xi}\varphi(\xi_{j}^{\nu}))\big|^{2}\right )^{\NNN} \left (1+\big |2^{\frac j2}(x- y+\nabla_{\xi}\varphi(\xi_{j}^{\nu}))'\big |^{2}\right )^{\NNN}}
}
for all multi-indices $\alpha$ and $\NNN\geq 0$, where $x'$ denotes $x-\pi_j^\nu(x)$, the orthogonal complement to the projection in the direction $\xi_j^\nu$. (See Lemma 3.2 in \cite{IRS} for the details.)

To make use of this decomposition we estimate the second term on the right-hand side of \eqref{eq: spliting the FIO on balls} by
\begin{equation}\label{Tphisigma2 over BC}
  \int_{B^{*c}}|T_{\sigma_{2}}^{\varphi}   a(x)|^p \, \dd x\leq \sum_{2^{j}<r^{-1}}\int_{B^{*c}}|T_{j}a(x)|^p \, \dd x +\sum_{2^{j}\geq r^{-1}}\int_{B^{*c}}|T_{j}a(x)|^p \,\dd x  .
\end{equation}
From \eqref{eq:claimkernel} it is possible to prove for $x\in B^{*c}$, any $\NNN$ and any $M\geq \left\lfloor n\brkt{\frac 1p -1}_+\right\rfloor$, one has the pointwise estimates 
  \begin{equation}\label{eq:operatorestimate Tnuj}
  \abs{T_j^\nu a(x)}\lesssim 
      \begin{cases}
       \frac{ 2^{j\brkt{m+\frac{n+1}2}}2^{j M}r^{ M}r^{n-\frac np}   }{\left (1+\big|2^{j}(x- \bar y+\nabla_{\xi}\varphi(\xi_{j}^{\nu}))_1\big|^{2}\right )^{\NNN} \left (1+\big|2^{\frac j2}(x- \bar y+\nabla_{\xi}\varphi(\xi_{j}^{\nu}))'\big|^{2}\right )^{\NNN}}\,,\, \qquad 2^j<r^{-1}\\
      \frac{2^{j\brkt{m+\frac {n+1}2}}  2^{-jM }r^{-M}r^{n-\frac np} 2^{4j\NNN}r^{4\NNN}}{\left (1+\big|2^{j}(x- \bar y+\nabla_{\xi}\varphi(\xi_{j}^{\nu}))_1\big|^{2}\right )^{\NNN}  \left (1+\big|2^{\frac j2}(x- \bar y+\nabla_{\xi}\varphi(\xi_{j}^{\nu}))'\big|^{2}\right )^{\NNN}}\, ,\,\qquad 2^j\geq r^{-1}
      \end{cases}
  \end{equation}
(See Lemma 3.4 in \cite{IRS} for the details.) For the first term on the right-hand side of \eqref{Tphisigma2 over BC} we use the first estimate of \eqref{eq:operatorestimate Tnuj} to deduce
\begin{equation*}
\int_{B^{*c}}|T_{j}a(x)|^p \dd x \lesssim 2^{j\frac{n-1}{2}} 2^{jp\brkt{m+\frac{n+1}2}}2^{j Mp}2^{-j\frac{n+1}2}r^{Mp}r^{np-n}.
\end{equation*}
(See Proposition 6.2 in \cite{IRS} for details.) Summing over $2^{j}<r^{-1} $ yields 
\m{\sum_{2^{j}<r^{-1}}\int_{B^{*c}}|T_{j}a(x)|^p \dd x \lesssim 1}
if $M$ and $\NNN$ are chosen appropriately. For the second term in \eqref{Tphisigma2 over BC} we have that $2^{j}\geq r^{-1}$, therefore using the second estimate of \eqref{eq:operatorestimate Tnuj} 
yields
\begin{equation*}
\int_{B^{*c}}|T_{j}a(x)|^p \dd x \lesssim 2^{j\frac{n-1}{2}}2^{jp\brkt{m+\frac {n+1}2}}  2^{-jMp }2^{-j\frac{n+1}{2}}r^{-Mp}r^{np-n} 
2^{4j\NNN p}r^{4\NNN p},
\end{equation*}
(See once again Proposition 6.2 in \cite{IRS} for details.) Summing over $2^{j}\geq r^{-1} $ yields 
\m{\sum_{2^{j}\geq r^{-1}}\int_{B^{*c}}|T_{j}a(x)|^p \dd x \lesssim 1}
for appropriate $M$ and $\NNN$, which concludes the proof for atoms supported on balls of radius less than or equal to one.

\subsubsection{Estimates of the norm on large balls}
When the atom is supported on a ball with radius greater than one we use a strategy developed by Ruzhanksy and Sugimoto~\cite{RuzhSug}. Once again we wish to show $\norm{T^\phase_{\sigma_2}a}_{L^p} \lesssim 1$, where $a$ is an atom, but this time supported on a ball of radius $r\geq1$. Without loss of generality, one can assume that this ball is centred at the origin. This is because the translation invariance of $L^p$ yields $\norm{T^\phase_{\sigma_2} a}_{L^p}= \norm{\tau_s^* T^\phase_{\sigma_2}\tau_s \tau_{-s} a}_{L^p}$, where $\tau_s$ is the operator of translation by $s\in\R^n$, and $\tau_s^* T^\phase_{\sigma_2}\tau_s$ is exactly the same operator as $T^\phase_{\sigma_2}$.

Following \cite{RuzhSug}, one introduces the function
\begin{equation}\label{defn:RS H function}
  H(z):=\inf_{\xi\in\R^n}\abs{z+\nabla\phase(\xi)},
\end{equation}
and its associated level sets
\[
\Delta_r:=\{z\in\R^n;\, H(z)\geq r\}.
\]
Clearly for $r_1\leq r_2$ one has  $\Delta_{r_1}\supseteq \Delta_{r_2}$ and
setting
\begin{align*}
&\MM:=\sum_{|\gamma|\leq n+1}\sup_{x,y,\xi\in\R^n}
|\partial^\gamma_\xi \sigma_2(\xi)\,\bra{\xi}^{{-m_c (p)}+|\gamma|}|,
\\
&\NN:=\sum_{1\leq|\gamma|\leq n+2}\sup_{\xi\in\R^n}
|\partial^\gamma_\xi\phase(\xi)\,\bra{\xi}^{-1+|\gamma|}|,
\end{align*}
it is easy to check that both $\MM$ and $\NN$ are finite due to the decay, support and homogeneity properties of $\sigma$ and $\phase$.

The following Lemmas \ref{Lem:outside RuzhSug} and \ref{Lem:kernel RuzhSug}, are special cases of Theorem 2.2  in \cite{RuzhSug}.
\begin{lem}\label{Lem:outside RuzhSug}
Let $r\geq1$.
Then we have
$\R^n\setminus\Delta_{2r}\subseteq\set{z:|z|< (2+\NN)\,r}$. Furthermore for $x\in\Delta_{2r}$ and $|y|\leq r$ one has
\begin{equation*}\label{H bounded by H}
 H(x)\leq 2H(x-y)
\end{equation*}
and therefore $x-y\in\Delta_r$
\end{lem}
\begin{lem}\label{Lem:kernel RuzhSug}
The kernel
\begin{equation*}\label{kernel of linear fio RuzSug}
  K(z)= \int_{\R^n}  \sigma_2 (\xi)\, e^{iz\cdot \xi +i\phase(\xi)}\, \ddd\xi.
\end{equation*}
 of $T^\phase_{\sigma_2}$ is smooth on $\bigcup_{r>0}\Delta_r$,
and for all $L>n$ it satisfies
\begin{equation*}\label{boundedness of HLK}
\norm{H^{L} K}
_{L^\infty(\R^n\times\R^n \times\Delta_r)}\leq C(r, L,\MM,\NN),
\end{equation*}
where $C(r,L,\MM,\NN)$ is a positive constant depending
only on $L$, $r>0$, $\MM$ and $\NN$.
For $0< p \leq 1$  and $L>n/p$, the function $H(z)$ satisfies the bound

\begin{equation*}\label{integrable HL}
\norm{ H^{-L}}_{L^p(\Delta_r)}
\leq C(r,L,\NN, p).
\end{equation*}
\end{lem}

Now returning to the problem of bounding the $L^p$-norm of $T^\phase_{\sigma_2} a$, we split
\nm{deltaandnotdelta}{
\left\| T^\phase_{\sigma_2} a \right\|_{L^p(\R^n)}\leq \left\| T^\phase_{\sigma_2} a\right\|_{L^p(\Delta_{2r})}+\left\| T^\phase_{\sigma_2} a\right\|_{L^p(\R^n \setminus \Delta_{2r})}.
}
We first estimate the integral in \eqref{deltaandnotdelta} over $\Delta_{2r}$. For $x\in \Delta_{2r}$ and $|y|\leq r$, Lemma \ref{Lem:outside RuzhSug} yields that
$H(x)\leq2 H(x-y)$ and $x-y\in\Delta_r$. This together with Lemma \ref{Lem:kernel RuzhSug} in turn imply that
\begin{align*}
| T^\phase_{\sigma_2} a(x)|
&\leq 2^{L}H(x)^{-L}\int_{|y|\leq r}
\abs{H(x-y)^{L}K(x-y)a(y)}\,\dd y
\\
&\leq 2^{L} H(x)^{-L}
\Vert H^{L} K\Vert_{L^\infty(\R^n \times \Delta_r)}
\Vert a\Vert_{L^1}\leq C(n,L,\MM,\NN)  H(x)^{-L},
\end{align*}
for $x\in\Delta_{2r}$, since $\Vert a\Vert_{L^1} \leq |B|^{1-1/p},$ and $r\geq 1$.
Therefore, choosing $L>n/p$, Lemma \ref{Lem:kernel RuzhSug} and the monotonicity
of $\Delta_r$ yield
\nma{awayfromdelta}{
\left\|T^\phase_{\sigma_2} a\right\|_{L^p(\Delta_{2r})}
&\leq
\left\|H(x)^{-L}\right\|_{L^p(\Delta_{2r})}
\leq C(n, \MM, \NN),
}
as required

For the integral in \eqref{deltaandnotdelta} over $\R^n \setminus \Delta_{2r}$, Lemma \ref{Lem:outside RuzhSug} and H\"older's inequality yield that
\begin{align*}
\left\| T^\phase_{\sigma_2} a\right\|_{L^p(\R^n \setminus \Delta_{2r})}
&\leq
|\R^n \setminus \Delta_{2r}|^{1-p/2}
\left\|T^\phase_{\sigma_2} a \right\|^{p}_{L^2(\R^n)}
\\
&\lesssim
r^{n(1-p/2)}\|a\|^{p}_{L^2(\R^n)}
\lesssim 1,
\end{align*}
which together with \eqref{awayfromdelta} proves the estimate $\norm{T^\phase_{\sigma_2}a}_{L^p} \lesssim 1$.

\section{A counter-example to the global $h^p$-boundedness of linear FIO's for $0 < p \leq n/(n+1)$}\label{Sharpness}

In Section \ref{sec:lowfreqlinearhp} we only succeeded in proving that the low-frequenct part of an FIO is bounded on $h^p$ for $p>n/(n+1)$. Here we shall constructively prove that the generic behaviour of an FIO acting on a Schwartz function is no better than $O(|x|^{-(n+1)})$ as $|x| \to \infty$ and so we cannot expect the boundedness of an FIO into $h^p \subseteq L^p$ for $p \leq n/(n+1)$ to hold. More specifically, for each dimension $n$, we will find a function $f \in \mathscr{S} \subseteq h^p$ for which
\begin{equation}\label{eq:ce}
T(f)(x) := \int_{\R^n} \widehat{f}(\xi)e^{ix\cdot\xi + i|\xi|}\ddd\xi = \left(\frac{\Gamma(\frac{n+1}{2})\widehat{f}(0)}{\pi^{(n+1)/2}i}\right)|x|^{-(n+1)} + O\left(\frac{1+\log |x|}{|x|^{n+3}}\right)
\end{equation}
as $|x| \to \infty$. The function $f$ will be chosen so that $\widehat{f}$ has compact support, thus showing that, regardless of the order of the decay of the amplitude, Theorem \ref{linearhpthm} cannot hold if $0 < p \leq n/(n+1)$. In the case $n=1$, this fact can also be proved directly, without the need for \eqref{eq:dimred} below, using integration by parts. A different proof, again in the case $n=1$, which yields the slightly stronger statement
\[
T(f)(x) = \left(\frac{\widehat{f}(0)}{\pi i}\right) \frac{1}{x^2} + O(x^{-4})
\]
as $|x|\to \infty$ can be found in \cites{IRS}.

We consider here the case $n > 1$. For a function $f_0 \colon \R^+ \to \C$, we can define a radial function $f \colon \R^n \to \C$ by $f(x) = f_0(|x|)$ for all $x \in \R^n$. The Fourier transform of this $f$ is then also a radial function and can be used to define a transformation on $f_0$, as
\[
\mathcal{F}_n(f_0)(r) := \widehat{f}(\xi)
\]
where $r = |\xi|$. For $n > 1$, the representation of the Fourier transform of a radial function (see, for example, \cite{SW}*{Chp~4, Thm~3.10}) together with properties of Bessel functions leads to the relation
\begin{equation}\label{eq:dimred}
\mathcal{F}_n(f_0) = -\frac{1}{2\pi}\mathcal{F}_{n+2}(f_1),
\end{equation}
for $f_1(r) = f_0'(r)/r$, provided $f_0$ is continuously differentiable and
\[
f_0(r) = \left\{
	\begin{array}{ll}
		O(r^{(1-n)/2})  & \mbox{as $r \to \infty$}; \\
		O(r^{-n})  & \mbox{as $r \to 0$}.
	\end{array}
\right.
\]

In order to prove \eqref{eq:ce} choose $f$ to be a smooth radial function whose Fourier transform $ \widehat{f}$ is compactly supported and equal to one in a neighbourhood of the origin. Furthermore, we set $g_0(r) = \widehat{f}(\xi)$ for $r = |\xi|$,
\begin{align*}
g_1(r) &= g_0(r)e^{ir}, \\
g_2(r) &= g_0(r)(e^{ir}-1-ir+r^2/2), \\
g_3(r) &= g_0(r)(1-r^2/2), \quad \mbox{and} \\
g_4(r) &= g_0(r)ir. \\
\end{align*}
Then $T(f)(x) = (2\pi)^{-n}\mathcal{F}_n(g_1)(|x|)$ and
\[
\mathcal{F}_n(g_1) = \mathcal{F}_n(g_2) + \mathcal{F}_n(g_3) + \mathcal{F}_n(g_4).
\]
Since $x \mapsto g_3(|x|)$ is smooth and compactly supported $\mathcal{F}_n(g_3)(r) = O(r^{-\NNN})$ as $r \to \infty$ for each $\NNN \in \N$. We introduce a smooth cut-off function $\chi$ which is equal to one on the unit ball supported in the double of the unit ball. Thus

\begin{align*}
\mathcal{F}_n(g_2)(|x|) &= \int_{\R^n} \widehat{f}(\xi)\left(e^{i|\xi|}-1-i|\xi|+\frac{|\xi|^2}{2}\right)e^{ix\cdot\xi}\, \dd\xi \\
&= \int_{\R^n} \widehat{f}(\xi)\chi(\xi/\lambda)\left(e^{i|\xi|}-1-i|\xi|+\frac{|\xi|^2}{2}\right)e^{ix\cdot\xi}\, \dd\xi\\
&+\quad \int_{\R^n} \widehat{f}(\xi)(1-\chi(\xi/\lambda))\left(e^{i|\xi|}-1-i|\xi|+\frac{|\xi|^2}{2}\right)e^{ix\cdot\xi}\, \dd\xi \\
&= A + B.
\end{align*}
To estimate $A$ and $B$ we can easily see that for $\xi \in \supp(\widehat{f})$ one has
\[
\left|\partial^\alpha_\xi(e^{i|\xi|}-1-i|\xi|+|\xi|^2/2)\right| \lesssim |\xi|^{3-|\alpha|}.
\]
Therefore $A \lesssim \lambda^{n+3}$, and for each $\NNN$
\begin{equation*}
\begin{split}
|B| = \left|\int_{\R^n} \widehat{f}(\xi)(1-\chi(\xi/\lambda))\left(e^{i|\xi|}-1-i|\xi|+\frac{|\xi|^2}{2}\right)\left[\frac{x\cdot\nabla_\xi}{2\pi i|x|^2}\right]^{\NNN}\left(e^{ix\cdot\xi}\right)\, \dd\xi\right| \\\lesssim \frac{1}{|x|^{\NNN}}\sum_{|\alpha_1| + |\alpha_2| + |\alpha_3| = {\NNN}} \int_{\R^n} \left|\partial^{\alpha_1}\widehat{f}(\xi)\right|\left|\partial^{\alpha_2}(1-\chi(\xi/\lambda))\right|\left|\partial^{\alpha_3}\left(e^{i|\xi|}-1-i|\xi|+\frac{|\xi|^2}{2}\right)\right| \, \dd\xi \\
\lesssim \frac{1}{|x|^{\NNN}} \sum_{|\alpha_1| + |\alpha_2| + |\alpha_3| = \NNN,\,|\alpha_2|>0} \lambda^{-|\alpha_2|}\int_{|\xi| \sim \lambda}|\xi|^{3-|\alpha_1|-|\alpha_3|} \, \dd\xi \\+\frac{1}{|x|^{\NNN}}\sum_{|\alpha_1| + |\alpha_3| = {\NNN}} \int_{\lambda < |\xi| \lesssim 1}|\xi|^{-|\alpha_1|+3-|\alpha_3|} \, \dd\xi,
\end{split}
\end{equation*}
whereby splitting the sum we can take advantage of the different support properties of $(1-\chi(\xi/\lambda))$ and its derivatives. Taking $\NNN = n+3$ we find
\[
|B| \lesssim \frac{1}{|x|^{n+3}} \left( 1-\log (\lambda)\right),
\]
therefore taking $\lambda = 1/|x|$ yields
\[
\mathcal{F}_n(g_2)(|x|) \leq |A| + |B| \lesssim \frac{1}{|x|^{n+3}} \left( 1+\log |x|\right),
\]
thus $\mathcal{F}_n(g_2)(r) = O((1+\log r)/r^{n+3})$.

To estimate $\mathcal{F}_n(g_4)$, we make use of \eqref{eq:dimred}. For this purpose we define
\begin{align*}
h_0(r) &= g_4'(r)/r, \\
h_1(r) &= ig_0'(r), \\
h_2(r) &= i(g_0(r)-1)/r, \quad \mbox{and} \\
h_3(r) &= i/r. \\
\end{align*}
Relation \eqref{eq:dimred} then gives us that
\[
\mathcal{F}_n(g_4) = -\frac{1}{2\pi}\mathcal{F}_{n+2}(h_0) = -\frac{1}{2\pi}\left(\mathcal{F}_{n+2}(h_1) + \mathcal{F}_{n+2}(h_2) + \mathcal{F}_{n+2}(h_3)\right).
\]
We have that $\mathcal{F}_{n+2}(h_1)(r) = O(r^{-\NNN})$ as $r \to \infty$ for each $\NNN \in \N$, since $h_1$ is smooth and compactly supported. It can also be shown that $\mathcal{F}_{n+2}(h_2)(r) = O(r^{-\NNN})$ as $r \to \infty$ for each $\NNN \in \N$, since $h_2$ is smooth and its higher-order derivatives decay sufficiently rapidly. Morover, $\mathcal{F}_{n+2}(h_3)(r) = i2^{n+1}\pi^{(n+1)/2}\Gamma(\frac{n+1}{2})r^{-(n+1)}$ (as can be found in, for example, \cite{SW}*{Chp~4, Thm~4.1}). Putting these together, we find that
\[
\mathcal{F}_n(g_4) = -\left(2^n\pi^{(n-1)/2}\Gamma\left(\frac{n+1}{2}\right)i\right)r^{-(n+1)} + O(r^{-\NNN})
\]
as $r \to \infty$ for each $\NNN \in \N$ and therefore we have proved \eqref{eq:ce} in the case $n>1$.

\section{The identification of the endpoint cases}\label{endpoint cases}

In order to prove Theorem \ref{thm:main} we wish to identify the various values of the exponents $p_1, p_2, \dots, p_N$ from which the general result claimed in Theorem \ref{thm:main} will follow via interpolation. These specific values are called \emph{endpoint cases} and to identify them we define the continuous convex piece-wise linear function
\nm{def:functionf}{
    F(x)=\norm{x-\frac{1}{2}\mathbf{1}}_{\ell^1}+\abs{x\cdot \mathbf{1}-\frac{1}{2}},\quad \mbox{(for $x\in \R^N$).}
}
where $\mathbf{1}=(1,\ldots,1)\in \R^{N}$. Bearing in mind that $p_0$ satisfies \eqref{eq:hoelder}, the right-hand side of \eqref{ineq:criticalexponent} can be written as
\m{
-(n-1)F(1/p_1,\ldots,1/p_N).
}
The fact we are restricting our attention to exponents $1 \leq p_j \leq \infty$ ($j=0,1,\dots,N$) means we are interested in the behaviour of $F$ on the domain
\nm{def:domaind}{
    D:=\set{x\in [0,1]^N:\, x\cdot \mathbf{1}\leq 1},
} 
and in understanding the set 
\[
    \set{(x,s)\in D\times [0,\infty):\, F(x)\leq s}.
\]
Since $F$ is convex and piece-wise linear, this set is a convex unbounded polytope. The extreme points of this set lie on the graph of $F$ over $D$ and are in one-to-one correspondence with the extreme points of the subsets of $D$ on which $F$ is linear. The subsets of $D$ on which $F$ is linear, as intersections of the compact convex set $D$ with convex sets (in this case half-spaces), are compact and convex. By the Krein-Milman theorem, these subsets of $D$ are the closed convex hull of their extreme points. Thus our task is to identify these convex sets and their extreme points. This is the content of the following theorem:

\begin{thm}\label{thm:extremepoints}
If $\{e_j\}_j$ is the standard basis in $\R^N$, then the set $D$ defined in \eqref{def:domaind} can be written as the union
\[
    D={D_0^0}\cup {D_0^1}\cup (\cup_{j=1}^N D_1^j),
\]
where
\ma{
D_0^0&=\mathrm{Hull}\brkt{\set{0} \cup \set{\frac{e_k}{2}}_{k=1}^N} \\
D_0^1&= \mathrm{Hull}\brkt{\set{\frac{e_k}{2}}_{k=1}^N \cup \set{\frac{e_k+e_\ell}{2}}_{k,\ell=1}^N}\\
D_{1}^{j} &=\mathrm{Hull}\brkt{\set{{e_j},\frac{e_j}{2}} \cup \set{\frac{e_j+e_k}{2}}_{k\neq j}}.
}
Moreover, $F$ defined in \eqref{def:functionf} is a linear function on each of these sets.
\end{thm}

Before proving Theorem \ref{thm:extremepoints} we observe that the values $(1/p_1,\dots,1/p_N)$ corresponding to the endpoint cases we need to consider are exactly the points of the set
\m{
\set{0} \cup \set{\frac{e_k}{2}}_{k=1}^N \cup \set{\frac{e_k+e_\ell}{2}}_{k,\ell=1}^N \cup \set{e_k}_{k=1}^N.
}
This leads to the following corollary.
\begin{cor}\label{cor:endpointcases}
It is enough to prove \emph{Theorem \ref{thm:main}} for the following values of exponents:
\begin{enumerate}[label=\emph{(}\roman*\emph{)}, ref=(\roman*)]
    \item \label{test1} $p_j=\infty$ for all $j=0,\ldots N$;
    \item $p_0=2$ and for any $1\leq j\leq N$, $p_{j}=2$, and $p_k=\infty$ for $k\neq j$;
    
    \item $p_0=1$ and any pair $1\leq j_1< j_2\leq N$,  $p_{j_1}=p_{j_2}=2$ and $p_k=\infty$ for $j_1\neq k\neq j_2$; and
    \item $p_0=1$ and for any $1\leq j\leq N$, $p_{j}=1$, and $p_k=\infty$ for $k\neq j$.
\end{enumerate}
\end{cor}
\begin{proof}
The proof is a fairly standard application of multilinear interpolation theory as described in \cite{GM}, using know results for interpolation spaces, for example Theorem 11 in \cite{MM}. 
\end{proof}

\begin{proof}[Proof of \emph{Theorem \ref{thm:extremepoints}}]
Let $N'=N'(x)$ denote the number of coordinates such that $x_j\geq 1/2$ (for $j=1,\dots,N$). That $x\in D$ means $\sum_{j=1}^N x_j \leq 1$, which in turn implies that $N'\in\set{0,1,2}$. We can therefore decompose $D=D_0\cup D_1\cup D_2,$
where, for each $k=0,1,2$,
\[
\mbox{$D_k$ is the closure of the set of points $x\in D$ for which  $N'(x)=k$.}
\]
We observe that $D_2$ consists exactly of the vertices $\frac{1}{2}(e_j+e_k)$ for $1\leq j<k\leq N$, and it is easy to check these points are limit points of $D_1$. Therefore $D_2\subset D_1$ and
\[
D=D_0\cup D_1.
\]

We can further decompose $$D_0=D_0^0\cup D_0^1,$$
where
\[
D_0^0 = \set{x\in D_0 : 0 \leq x\cdot\mathbf{1} \leq \frac{1}{2}} \quad \mbox{and} \quad D_0^1 = \set{x\in D_0 : \frac{1}{2}\leq x\cdot\mathbf{1}\leq 1}
\]
Since $0\leq x\cdot\mathbf{1}< \frac{1}{2}$ and $x_j \geq 0$ for $x = (x_1,\dots,x_N) \in D_0^0$, all points $x \in D_0^0$ can be expressed as the convex hull of the points $0$, and $\frac{1}{2}e_k$, for $k=1,\ldots, N$. So, ${D_0^0}=\mathrm{Hull}(\set{0,\frac{e_1}{2},\dots, \frac{e_N}{2}})$.

Leaving $D_0^1$ for a moment, we next consider $D_1$. We can write
\[
D_1=\cup_{j=1}^n  D_{1}^{j},
\]
where
\[
    D_{1}^{j}=\set{x \in D_1 : \mbox{$x_j\geq \frac{1}{2}\geq x_k$ for all $k\neq j$}}. 
\]
Note that $D_{1}^{j}$ is the translation of $D_0^0$ by $\frac{e_j}{2}$, so it follows that
\[
    \begin{split}
        D_{1}^{j} &=\frac{e_j}{2}+D_0^0\\
        &=\frac{e_j}{2}+\mathrm{Hull}(\set{0,\frac{e_1}{2},\dots, \frac{e_N}{2}})\\
        &=\mathrm{Hull}\brkt{\set{{e_j},\frac{e_j}{2}} \cup \set{\frac{e_j+e_k}{2}}_{k\neq j}}.
    \end{split}
\]

We now return to $D_0^1$. Given a fixed arbitrary point $x \in D_0^1$ consider the maximal line segment contained in the ray from the origin through $x$ which is contained in $D_0^1$. This is a set of the form
\[
\set{y = \lambda x : \lambda_- \leq \lambda \leq \lambda_+}.
\]
The factor $\lambda_-$ will be determined by when the ray breaks through the plane $y \cdot \mathbf{1} = 1/2$, so $\lambda_-$ solves the equation $\lambda_- x \cdot \mathbf{1} = 1/2$, and $\lambda_+$ will be determined by when the ray first breaks through one of the planes $y \cdot e_j = 1/2$ ($j=1,\dots,N$) and $y \cdot \mathbf{1} = 1$, therefore
\[
\lambda_+ = \min\set{\lambda_j : \mbox{$\lambda_j x \cdot e_j = 1/2$ ($j=1,\dots,N$) and $\lambda_{N+1} x \cdot \mathbf{1} = 1$}}.
\]
However
\ma{
\lambda_-x &\in \mathrm{Hull}(\set{\frac{e_1}{2},\dots, \frac{e_N}{2}}), \\
\lambda_jx &\in \mathrm{Hull}\brkt{\set{\frac{e_j}{2}} \cup \set{\frac{e_j+e_k}{2}}_{k\neq j}} \quad \mbox{if $\lambda_+ = \lambda_j$ for $j=1,\dots,N$, and} \\
\lambda_{N+1}x &\in \mathrm{Hull}\brkt{\set{\frac{e_k+e_\ell}{2}}_{k\neq \ell}}, \quad \mbox{if $\lambda_+ = \lambda_{N+1}$,}
}
so it follows that
\m{
x \in \mathrm{Hull}\brkt{\set{\frac{e_k}{2}}_{k=1}^N \cup \set{\frac{e_k+e_\ell}{2}}_{k,\ell=1}^N}.
}

Summarising, we can write
\[
    D={D_0^0}\cup {D_0^1}\cup (\cup_{j=1}^N D_1^j),
\]
where each set is convex and the extreme points are the ones given in the statement of Theorem \ref{thm:extremepoints}. 

We now check that $F$ is linear on these sets. For $x\in D_0$, $x_k\leq \frac{1}{2}$ for all $k$, so we have that
\ma{
    F(x) &= \sum_{k=1}^N \brkt{\frac12-x_k} + \abs{\sum_{k=1}^N x_k - \frac12} \\
    &=
    \begin{cases}
        \frac{N-1}{2} &\mbox{if $x\in D_0^1,$}\\
        \frac{N-1}{2}+\brkt{1-2x\cdot \mathbf{1}}&\mbox{if $x\in D_0^0.$}
    \end{cases}
}
Now if $x\in D_1^j$, then $x_k\leq \frac{1}{2}\leq x_j$ for all $k\neq j$ and so $x\cdot \mathbf{1} \geq 1/2$. Thus we can write
\ma{
F(x) &= \sum_{k\neq j} \brkt{\frac12-x_k} + \brkt{x_j-\frac12} + \sum_{k=1}^N x_k - \frac12 \\
&= \frac{N-1}2 + 2x_j - 1.
}
\begin{figure}
    \centering
    \begin{minipage}{.45\textwidth}
    \centering
    \includegraphics[scale=.4]{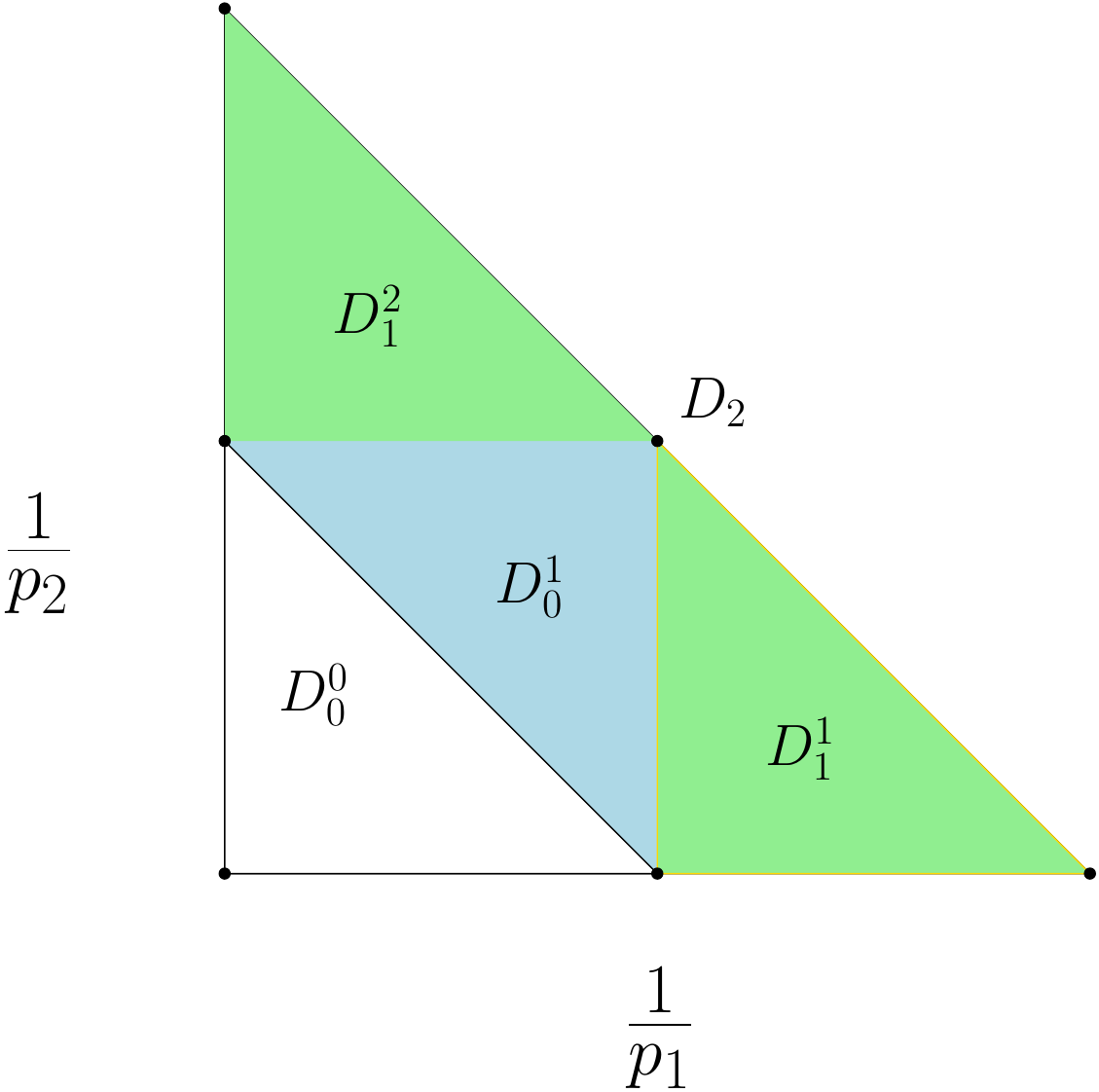}
\end{minipage}\begin{minipage}{.45\textwidth}
    \centering
   \includegraphics[scale=.2
   ]{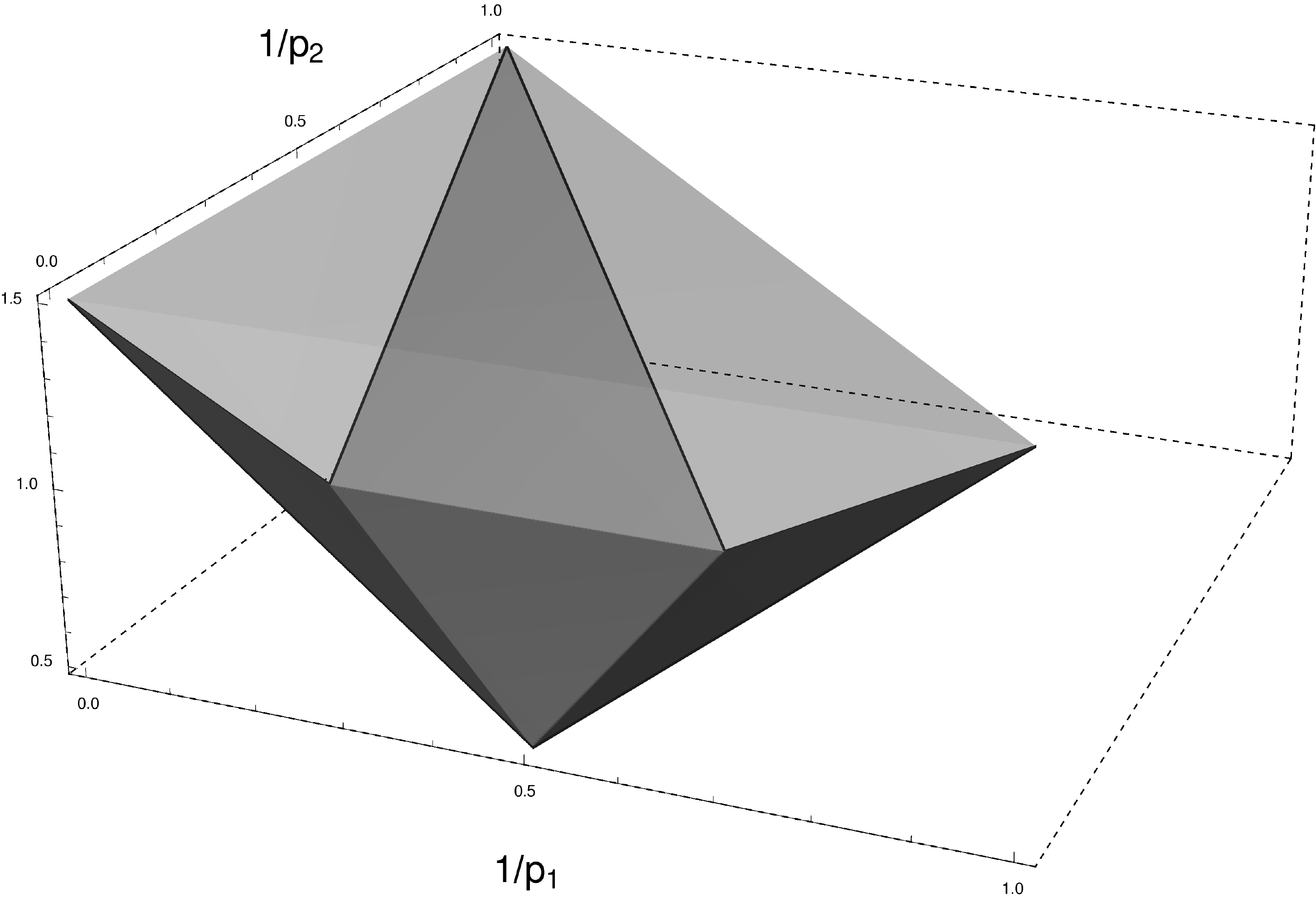}
\end{minipage}
    \caption{Decomposition of $D$ and graph of $F$ for $N=2$.}
    \label{fig:my_label}
\end{figure}

Thus, in all cases, we see that $F$ is linear on each of the convex sets 
${D_0^0}$, ${D_0^1}$ and $D_1^j$, for $j=1,\ldots, N$. 
\end{proof}

\section{Frequency decomposition of multilinear FIOs}\label{sec:freqdecom}

In what follows we shall demonstrate that the regularity of $T^\Phase_\sigma$ can be obtained by considering three frequency regimes: When $\Xi$ lies inside a compact set; when one component of $\Xi= (\xi_1, \dots, \xi_N)$ dominates the others; and when two fixed components of $(\xi_1, \dots, \xi_N)$ are comparable to each other.

In all that follows we take $N>1$. First we define the component of $\sigma$ with frequency support contained in a compact set. We introduce the cut-off function $\chi \colon \R^{nN} \to \R$, such that $\chi(\Xi) = 1$ for $|\Xi| \leq 1/8$ and $\chi(\Xi) = 0$ for $|\Xi| \geq 1/4$ and define
\[
\sigma_0(x,\Xi) = \chi(\Xi)\,\sigma(x,\Xi).
\]
To define the components of $\sigma$ where one frequency dominates all the others, we construct a cut-off function $\nu \colon \R^{nN} \to \R$ such that $\nu(\Xi) = 0$ for $|\xi_1| \leq 32\sqrt{N-1}\,|\Xi'|$ and $\nu(\Xi) = 1$ for $64\sqrt{N-1}\,|\Xi'| \leq |\xi_1|$, where $\Xi' := (\xi_2,\ldots,\xi_N)$. 
This can be done by taking $\lambda\in \mathcal{C}^\infty(\R)$ such that $\lambda(t)=1$ if $t\leq c_1$ and $\lambda(t)=0,$ if $t\geq c_2$ for two real numbers $0<c_1<c_2<1$ which will be decided momentarily. Define 
\[
    \nu(\Xi)=1-\lambda\brkt{\frac{\abs{\xi_1}^2}{\abs{\Xi}^2}}\in\mathcal{C}^\infty(\R^{nN}\setminus 0).
\]
By construction, it follows that
\[
    \nu(\Xi)=\begin{cases}
        0 &\mbox{if\,} \abs{\xi_1}^2\leq c_1 \abs{\Xi}^2\\
        1 & \mbox{if\,} \abs{\xi_1}^2\geq c_2 \abs{\Xi}^2
    \end{cases}=\begin{cases}
        0 &\mbox{if\,} \abs{\xi_1}\leq \sqrt{\frac{c_1}{1-c_1}} \abs{\Xi'}\\
        1 & \mbox{if\,} \abs{\xi_1}\geq \sqrt{\frac{c_2}{1-c_2}} \abs{\Xi'}. 
    \end{cases}
\]
and a calculation shows that taking 
\nma{cone ctwo}{
    c_1&=1-\frac{1}{1+32^2(N-1)}= \frac{2^{10}(N-1)}{1+2^{10}(N-1)},\quad \mbox{and} \\ c_2&=1-\frac{1}{1+4\cdot 32^2(N-1)}= \frac{2^{12}(N-1)}{1+2^{12}(N-1)},
}
ensures we obtain the function $\nu$ with the required properties. Given $j=1,\ldots N$ we define $\Xi'_j:=(\xi_1,\ldots, \xi_{j-1},\xi_{j+1},\ldots ,\xi_N)$ and 
\[
    \nu_j(\Xi):=\nu(\xi_j,\Xi'_j),
\]
for all $\Xi\in \R^{nN}$. We then define the component of $\sigma$ for which $\xi_j$ dominates the other frequency components to be
\[
\sigma_j(x,\Xi) = (1-\chi(\Xi))\,\nu_j(\Xi)\,\sigma(x,\Xi), \quad \mbox{for $j=1,\ldots N$.}
\]

What remains of $\sigma$ will be split into functions on whose support two frequency components are comparable. Observe that the supports of the $\nu_j$ are disjoint, therefore the $\Xi$-support of
\nm{eq:amp_remainder}{
\sigma(x,\Xi) -\sum_{j=0}^N \sigma_j(x,\Xi)
}
is contained in the set of all $\Xi$ for which no $\nu_j(\Xi)$ is equal to one. We define
\[
    \widetilde{\nu}(\Xi)=1-\lambda\brkt{\frac{\abs{\xi_1}^2}{c_3\abs{\Xi}^2}}\in\mathcal{C}^\infty(\R^{nN}\setminus 0).
\]
for some constant $0 < c_3 < 1$ (to be chosen momentarily) and
\[
    \widetilde{\nu}_j(\Xi):=\widetilde{\nu}(\xi_j,\Xi'_j).
\]
For fixed $k$, if $\Xi$ is not contained in the support of $\widetilde{\nu}_j$ for any $j\neq k$, then $\abs{\xi_j}^2 \leq c_3c_1\abs{\Xi}^2$ for all $j \neq k$ and consequently
\[
\abs{\xi_k}^2 \geq (1-c_3c_1(N-1))\abs{\Xi}^2.
\]
Thus, we choose $c_3$ so that $1-c_3c_1(N-1) > c_2$, and all $\Xi$ which are not contained in the support of $\widetilde{\nu}_j$ for any $j\neq k$ will be such that $\nu_k(\Xi) = 1$. Therefore the functions
\[
\Theta_{j,k}(\Xi) := \frac{\widetilde{\nu}_j(\Xi)\widetilde{\nu}_k(\Xi)}{\brkt{\sum_{\ell=1}^N \widetilde{\nu}_\ell(\Xi)}^2}
\]
are a smooth partition of the $\Xi$-support of \eqref{eq:amp_remainder} and $\abs{\xi_j} \approx \abs{\xi_k}$ on the support of $\Theta_{j,k}$. Defining
\[
\sigma_{j,k}(x,\Xi) = (1-\chi(\Xi))\Theta_{j,k}(\Xi)\brkt{\sigma(x,\Xi) -\sum_{\ell=0}^N \sigma_\ell(x,\Xi)}, \quad \mbox{for $j,k=1,\ldots N$}
\]
we have completed our decomposition of the amplitude $\sigma$ as
\m{
\sigma(x,\Xi) = \sigma_0(x,\Xi) + \sum_{j=1}^N \sigma_j(x,\Xi) + \sum_{j\neq k} \sigma_{j,k}(x,\Xi),
}
where $\sigma_0$ has compact $\Xi$-support, $\abs{\xi_j}$ dominates $\abs{\Xi}$ on the $\Xi$-support of $\sigma_j$, and $\abs{\xi_j} \approx \abs{\xi_k}$ on the $\Xi$-support of $\sigma_{j,k}$.

It is easy to check that if $\sigma \in S^m(n,N)$ then $\sigma_j$ and $\sigma_{j,k}$ are also in $S^m(n,N)$ for all $j,k=1,\dots,N$ and $\sigma_0 \in S^\mu(n,N)$ for all $\mu \in \R$.

\section{Boundedness results for $T_{\sigma_{j}}^\Phase$}\label{sec:sigmaj}

We will restrict our discussion to the amplitude $\sigma_1$. This will be sufficient for the treatment of an arbitrary $\sigma_j$ since a permutation of the frequency variables $\xi_1,\dots,\xi_N$ reduces the boundedness of $\sigma_j$ in one of the endpoint cases from Corollary \ref{cor:endpointcases} to an endpoint case for $\sigma_1$.

We begin by decomposing $\sigma_1$ in a similar fashion to Coifman and Meyer \cite{CM4}. The rough idea is to first introduce a Littlewood-Paley partition of unity in the $\xi_1$ variable. One can then make use of the fact $\abs{\Xi} \lesssim \abs{\xi_1}$ on the $\Xi$-support of $\sigma_1$ to see that, for each term in the Littlwood-Paley decomposition, one can introduce for free a second Littlewood-Paley cut off function in the variable $\xi_1 + \dots + \xi_N$ (that is, the ``dual'' frequency variable). The same support property allows one to also introduce low-frequency cut-off operators (written as $P_k^{u_j}$ below) in each of the $\xi_j$-variables ($j=2,\dots,N$) which restrict $\abs{\xi_j} \lesssim 2^k$ when $\abs{\xi_1} \approx 2^k$. For this purpose it is more useful to have that the squares of the functions form a partition of unity than the functions themselves, that is \eqref{eq:littlewoodpaleysquare} below holds instead of \eqref{eq:littlewoodpaley}. So although the $\psiF_k$ in following construction are essentially a Littlewood-Paley partition of unity in the sense of Definition \ref{def:LP}, we emphasis that they depart slightly from the definition there.

We introduce a positive, radial, radially decreasing, smooth cut-off function $\vth \colon \R^n \to \R$ which satisfies $\vth(\xi) = 1$ if $|\xi| \leq 1$ and $\vth(\xi) = 0$ if $|\xi| \geq 2$ and define the non-negative functions $\th_k$, $\psiF_k$ and $\phiF_k$ via the relations
\begin{itemize}

\item $\th_k(\xi) := \vth(2^{3-k}\xi)$,

\item $\psiF_k(\xi)^2 := \vth(2^{-1-k}\xi)^2 - \vth(2^{2-k}\xi)^2$, and

\item $\phiF_k(\xi)^2 := \vth(2^{-3-k}\xi)^2 - \vth(2^{4-k}\xi)^2$.

\end{itemize}
Using the support properties of these functions, it is easy to verify the following facts:

\begin{enumerate}

\item \label{psisupport1} $\psiF_k(\xi) = 1$ for $2^{k-1} \leq |\xi| \leq 2^{k+1}$;

\item \label{psisupport2} $\psiF_k(\xi) = 0$ for $|\xi| \leq 2^{k-2}$ and $2^{k+2} \leq |\xi|$;

\item $\th_k(\eta) = 1$ for $|\eta| \leq 2^{k-3}$;

\item $\th_k(\eta) = 0$ for $2^{k-2} \leq |\eta|$;

\item $\phiF_k(\xi+\eta) = 1$ for $2^{k-3} \leq |\xi+\eta| \leq 2^{k+3}$; 

\item $\phiF_k(\xi+\eta) = 0$ for $|\xi+\eta| \leq 2^{k-4}$ and $2^{k+4} \leq |\xi+\eta|$.

\end{enumerate}
Given the support properties of $\sigma_1$, it follows that if $\psiF_k(\xi_1) \neq 0$ and $\sigma_1(x,\Xi) \neq 0$, then
\m{
|2^{-k}\Xi_1'| \leq \frac{|2^{-k}\xi_1|}{32\sqrt{N-1}} \leq \frac{2^{2}}{32\sqrt{N-1}} = \frac{2^{-3}}{\sqrt{N-1}} \quad 
}
which implies that $\th_k(\xi_j) = 1$ for $j=2,\ldots, N$.

Likewise, when $\psiF_k(\xi_1) \neq 0$ and $\sigma_1(x,\Xi) \neq 0$, then
\nma{ineqforcut-off}{
\frac{1}{4} - \frac{1}{8} &\leq |2^{-k}\xi_1| - \sqrt{N-1}|2^{-k}\Xi_{1}'| \\
&\leq |2^{-k}(\xi_1+\dots+\xi_N)| \leq |2^{-k}\xi| + \sqrt{N-1}|2^{-k}\Xi_{1}'| \leq 4 + \frac{1}{8} < 8 \\
&\quad \mbox{which implies} \quad \phiF_k(\xi_1+\dots+\xi_N) = 1.
}
Observe that on the support of $\sigma_1$, 
\[
    \abs{\Xi}^2=\abs{\xi_1}^2+\abs{\Xi_1'}^2\geq 1/64,\qquad \abs{\xi_1}^2 \geq 32^2(N-1) \abs{\Xi_1'}^2.
\]
Then 
\[
    1/64\leq \brkt{1+\frac{1}{ 32^2(N-1)}}\abs{\xi_1}^2, 
\]
and so
\[
    \abs{\xi_1}^2\geq \frac{16(N-1)}{1+32^2(N-1)}>0.
\]
Finally, it follows directly from the definition above that each function $\psiF_k$ is radial, real-valued, and
\nm{eq:littlewoodpaleysquare}{
\sum_{k=-\infty}^\infty  \psiF_k(\xi)^2 = 1 \quad \mbox{for all $\xi \neq 0$.}
}

Using these facts, there exists $k_0\in \Z$ (independent of $x$) such that we can write $T^\Phase_{\sigma_1}$ as
\nma{decomp12}{
&T^\Phi_{\sigma_1}(f_1,\ldots,f_N)(x) \\
&= \int_{\R^{nN}} \sum_{k\geq k_0} \psiF_k(\xi_1)^2 \prod_{j=2}^N\th_k(\xi_j)^2 \phiF_k(\xi_1+\ldots+\xi_N)^2 \sigma_1(x,\Xi) \widehat{f_1}(\xi_1)\prod_{j=2}^N\widehat{f_j}(\xi_j)\, e^{ix\cdot(\xi_1 +\dots+\xi_N)} e^{i\Phase(\Xi)} \ddd\Xi
}
which by setting 
\begin{equation}\label{the complete phase}
\Phi(x, \Xi):= x\cdot(\xi_1 +\dots+ \xi_N)+ \phase_0(\xi_1+\dots+\xi_N) + \sum_{j=1}^N\phase_j(\xi_j),
\end{equation}
can in turn be written as
\nma{decomp11}{
& \int_{\R^{nN}} \sum_{k\geq k_0}^\infty {\mathfrak a}(k,x,2^{-k}\Xi)\left[|\xi_1+\ldots+\xi_N|^{m_{0}}\phiF_k(\xi_1+\ldots+\xi_N)\right] \times \\
& \quad \left[|\xi_1|^{m_1}|2^{-k}\xi_1|^{m-m_0-m_1}\psiF_k(\xi_1)\widehat{f}(\xi_1)\right] \left[\prod_{j=2}^N2^{km_j}\th_k(\xi_j)^2 \widehat{f_j}(\xi_j)\right]  e^{i\Phase(x,\Xi)}  \ddd\Xi
}
where $m = \sum_{j=0}^N m_j$, and
\m{
{\mathfrak a}(k,x,\Xi) = \sigma_1(x,2^{k}\Xi)\,\psiF_1(\xi_1)\prod_{j=2}^N\th_1(\xi_j)\phiF_1(\xi_1+\ldots+\xi_N)\left(\frac{2^{-k}}{|\xi_1|}\right)^{m-m_0}\left(\frac{2^{-k}}{|\xi_1+\ldots+\xi_N|}\right)^{m_0}.
}
If we introduce a high frequency cut-off $\chi_0$ that satisfies 
\begin{itemize}

\item $\chi_0(\xi) = 1$ for $|\xi| \geq 2^{k_0-4}$ and
\item $\chi_0(\xi) = 0$ for $|\xi| \leq 2^{k_0-5}$,
\end{itemize}
we can use \eqref{ineqforcut-off}, and \eqref{psisupport1} and \eqref{psisupport2} above to rewrite \eqref{decomp11} as
\ma{
& \int_{\R^{nN}} \sum_{k\geq k_0}^\infty {\mathfrak a}(k,x,2^{-k}\Xi)\left[|\xi_1+\ldots+\xi_N|^{m_{0}}\chi_0(\xi_1+\ldots+\xi_N)\phiF_k(\xi_1+\ldots+\xi_N)\right] \times \\
& \quad \left[|\xi_1|^{m_1}|2^{-k}\xi_1|^{m-m_0-m_1}\chi_0 (\xi_1)\psiF_k(\xi_1)\widehat{f}(\xi_1)\right] \left[\prod_{j=2}^N2^{km_j}\th_k(\xi_j) \widehat{f_j}(\xi_j)\right]  e^{i\Phase(x,\Xi)}  \ddd\Xi.
}

Making use of the Fourier inversion formula, we can write
\m{
{\mathfrak a}(k,x,\Xi) = \int \frac{{\mathfrak m}(k,x,U)}{(1+|U|^2)^M}\, e^{i\,\Xi\cdot U}\, \ddd U,\qquad U=(u_1,\ldots,u_N),
}
for a smooth bounded function ${\mathfrak m}$. This means we can then write $T^\Phi_{\sigma_1}(f_1,\ldots,f_N)(x)$ as a weighted average in $U$ of
\begin{equation}
\label{t_main representation formula}
\begin{split}
&\sum_{k=k_0}^\infty {\mathfrak m}(k,x,U)\int \left[|\xi_1+\ldots+\xi_N|^{m_0}\phiF_k(\xi_1+\ldots+\xi_N)\chi_0(\xi_1+\ldots+\xi_N)\right] \times \\
& \quad \left[|\xi_1|^{m_1}|2^{-k}\xi_1|^{m-m_0-m_1}\chi_0 (\xi_1)\psiF_k(\xi_1)e^{i2^{-k}\xi_1\cdot u_1}\widehat{f_1}(\xi)\right] \times \\
& \quad    \left[\prod_{j=2}^N2^{km_j}\th_k(\xi_j) \widehat{f_j}(\xi_j)e^{i2^{-k} \xi_j\cdot u_j}\right]  e^{i\Phase(x,\Xi)}  \ddd\Xi.
\end{split}
\end{equation}
Finally we can write \eqref{t_main representation formula} as
\nm{t_piecedecomp}{
 B(f_1,\ldots,f_N)(x)=\sum_{k=k_0}^\infty M_{\mathfrak m}\circ T^{\phase_0}_{b_0} \circ Q_{k}^0\left[(Q_{k}^{u_1} \circ T^{\phase_1}_{b_1})(f_1)\prod_{j=2}^N(P_{k}^{u_j} \circ T^{\phase_j}_{b_j})(f_j)\right](x),
}
where
\begin{align*}
\widehat{Q_k^0(f)}(\xi) &= \phiF_k(\xi)\widehat{f}(\xi), & b_0(\xi) &= |\xi|^{m_0}\chi_0(\xi), \\
\widehat{Q_k^{u_1}(f)}(\xi) &= |2^{-k}\xi|^{m-m_0-m_1}\psiF_k(\xi)e^{i2^{-k} \xi\cdot u_1}\widehat{f}(\xi), & b_1(\xi) &= |\xi|^{m_1}\chi_0 (\xi), \\
\widehat{P_k^{u_j}(f)}(\xi) &= \th_k(\xi)e^{i2^{-k} \xi\cdot u_j}\widehat{f}(\xi), & b_j(k,\xi) &= 2^{km_j}\omegaF_k(\xi),
\end{align*}
for $j=2,\dots,N$, $\omegaF_k(\xi) := \th_k(\xi/2)$ is a bump function equal to one on the support of $\th_k$, and $M_{\mathfrak m}$ denotes multiplication by ${\mathfrak m}$.\footnote{\label{abiguity}The notation $Q^0_k$ and $Q^{u_1}_k$ is potentially ambiguous as $Q^{u_1}_k|_{u_1=0}$ is not the same operator as $Q^0_k$. However, in practice no confusion need arise, so to avoid a profusion of notation, we tolerate this imprecision.}

The position of the operator $M_{\mathfrak m}$ and the fact that ${\mathfrak m}$ depends on both $k$ and $x$, causes problems if we wish to make use of various square function and Carleson measure estimates to estimate norms of \eqref{t_piecedecomp}. We can overcome the problems by observing that this dependency is in fact periodic. Indeed, since $Q_k^0 = (Q_{k-1}^0 + Q_k^0 + Q_{k+1}^0)\circ Q_k^0$ and $Q_k^0 \circ Q_{k'}^0 \equiv 0$ if $|k-k'| \geq 2$, we can write
\m{
M_{{\mathfrak m}} \circ T^{\phase_0}_{b_0} \circ Q_{k}^0
= \left(\sum_{k'=k-1}^{k+1} T_{k,k'}^{U}\right) \circ Q_{k}^0 \\
= \left(\sum_{\ell=-1}^1\sum_{k'-k\equiv\ell\Mod{3}} T_{k'+\ell,k'}^{U}\right) \circ Q_{k}^0
}
where $T_{j,k}^{U}$ is the FIO with amplitude ${\mathfrak m}(j,x,U)\, b_0 (\xi)\,\phiF_{k}(\xi)$ and phase $\phase_0$. Observe that
\[
\mathcal{T}_{k}^U:=\sum_{\ell=-1}^{1}\sum_{k'-k\equiv\ell\Mod{3}} T_{k'+\ell,k'}^{U}
\]
is periodic in $k$ with period $3$, and is an FIO with amplitude in $S^{m_0}$. 
Thus \eqref{t_piecedecomp} can be rewritten as
\m{
    \sum_{\ell=0}^2 \mathcal{T}_\ell^{U} \left({B}_\ell(f_1,\ldots,f_N)\right)(x),
}
where
\nma{the magic operator}{
    & {B}_\ell(f_1,\ldots,f_N)(x) \\
    &:={ \sum_{k\equiv \ell\Mod{3},\,k\geq k_0 } \chi_0(2D)\,Q_{k}^0  \left[(Q_{k}^{u_1} \circ T^{\phase_1}_{b_1})(f_1)\prod_{j=2}^N(P_{k}^{u_j} \circ T^{\phase_j}_{b_j})(f_j)\right](x),}}
and $\chi_0$ is the same high-frequency cut-off introduced above (and is a symbol belonging to $S^0$). Now, by Theorem \ref{linearhpthm}, each $\mathcal{T}_\ell^{U}$ is a bounded operator on $X^p$ (with norms uniform in $U$) and so the boundedness of $T_{\sigma_1}^\Phase$ is reduced to studying the boundedness of $B_\ell$. In the remainder of this section, we prove this boundedness in each of the endpoint cases from Corollary \ref{cor:endpointcases}. Due to the symmetry of \eqref{the magic operator} in the indicies $j=2,\dots,N$ we only need to consider endpoint cases $(p_0,\dots,p_N)$ which are distinct within the equivalence class of permutations of $(p_2,\dots,p_N)$.

\subsection{Boundedness with the target space $L^2$.} 

In this case we take $m_0=0$. By duality and \eqref{the magic operator} we have
\ma{
& \quad \norm{{B}_\ell(f_1,\ldots,f_N)}_{L^2} \\
&= \sup_{\|f_0\|_{L^2}=1} \abs{\int f_0(x){B}_\ell(f_1,\ldots,f_N)(x) \dd x} \\
&= \sup_{\|f_0\|_{L^2}=1} \left|\sum_{k\equiv \ell\Mod{3},\,k\geq k_0 } \int Q_{k}^0 (\chi_0(2D)f_0)(x)(Q_{k}^u \circ T^{\phase_1}_{b_1})(f_1)(x)\prod_{j=2}^N(P_{k}^{u_j} \circ T^{\phase_j}_{b_j})(f_j)(x) \, \dd x\right| \\
&\leq \sup_{\|f_0\|_{L^2}=1} \left(\sum_{k\equiv \ell\Mod{3},\,k\geq k_0 } \int |Q_{k}^0(\chi_0(2D)f_0)(x)|^2 \,\dd x\right)^{1/2} \times \\
&\quad \left(\sum_{k\equiv \ell\Mod{3},\,k\geq k_0 } \int |(Q_{k}^u \circ T^{\phase_1}_{b_1})(f_1)(x)\prod_{j=2}^N(P_{k}^{u_j} \circ T^{\phase_j}_{b_j})(f_j)|^2\, \dd x\right)^{1/2}.
}
For the first factor above we just use the quadratic estimate
\m{
\left(\sum_{k\equiv \ell\Mod{3},\,k\geq k_0 } \int |Q_{k}^0(\chi_0(2D)f_0)(x)|^2 \, \dd x \right)^{1/2} \lesssim \norm{\chi_0(2D)f_0}_{L^2}\lesssim 1.
}
Thus it remains to control
\nm{t_l2targetsecondfactor}{
\left(\sum_{k\equiv \ell\Mod{3},\,k\geq k_0 } \int |(Q_k^{u_1}\circ T^{\phase_1}_{b_1})(f_1)(x)\prod_{j=2}^N(P_{k}^{u_j} \circ T^{\phase_j}_{b_j})(f_j)|^2\, \dd x\right)^{1/2},
}
and precisely how this is done depends on the endpoint case considered, so we consider each case in turn. 
\subsubsection{$\bmo\times\dots\times\bmo\times L^2 \to L^2$} 

Here we take $n\geq 2,$ $m_j= -(n-1)/2$, $f_j \in \bmo$ 
for $j=1,\ldots, N-1$, $m_N=0$ and $f_N \in L^2$. By Theorem \ref{linearhpthm} we know that $T^{\phase_1}_{b_1}(f_1) \in \mathrm{BMO}$ when $f_1 \in \bmo$. This implies that
\m{\dd \mu(x,t) = \sum_{k\in \Z} |(Q_{k}^{u_1}\circ T^{\phase_1}_{b_1})(f_1)(x)|^2 \dd x\, \delta_{2^{-k}}(t),}
where $\delta_{2^{-k}}$ is a Dirac mass at the point $2^{-k}$, is a Carleson measure with the Carleson norm bounded by a constant multiple of $\Vert f_1 \Vert_{\bmo}$. Moreover, the non-tangential maximal function of $(x,t) \mapsto (P_{k}^{u_N} \circ T^{\phase_N}_{b_N})(f_N)(x)\delta_{2^{-k}}(t)$ is in $L^2$ when $f_N \in L^2$. Thus, to control \eqref{t_l2targetsecondfactor} with \eqref{ineq:carll2} and conclude the proof in this endpoint case, it is enough to apply \eqref{eq:uniform_bounddness} from the following lemma to $P_{k}^{u_j} \circ T^{\phase_j}_{b_j}$ for each $j=2,\ldots, N-1$. 

\begin{lemma}\label{lem:Wase-Lin} Let 
\[
    m=-(n-1)\abs{\frac{1}{p}-\frac{1}{2}},\qquad n/(n+1)<p\leq \infty.
\]
Let 
\[
    b(k,\xi) = 2^{km}\omegaF_k(\xi),\qquad \widehat{P_k^{u}(g)}(\xi)=\th_k(\xi)e^{i2^{-k} \xi\cdot u}\widehat{g}(\xi), 
    \]
where $\omegaF_k$ and $\th_k$ are the cut-off functions defined above.
It follows that
\begin{equation}\label{eq:hp_composition}
\sup_k \norm{(P_{k}^{u} \circ T^{\phase}_{b})(f)}_{h^p}\lesssim \norm{f}_{h^p},
\end{equation}
and for $n\geq 2$ one also has for $m=-(n-1)/2$
\begin{equation}\label{eq:uniform_bounddness}
    \sup_k \norm{(P_{k}^{u} \circ T^{\phase}_{b})(f)}_{L^\infty}\lesssim \norm{f}_{\bmo},\qquad \mbox{and} \qquad \sup_k \norm{(P_{k}^{u} \circ T^{\phase}_{b})(f)}_{h^1}\lesssim \norm{f}_{L^1}.
\end{equation}
\end{lemma}
\begin{proof}
The proof of \eqref{eq:hp_composition} follows from the fact that the amplitude of $P_{k}^{u} \circ T^{\phase}_{b}$ is in $S^m$ uniformly in $k$. By duality, self-adjointness of the  operators involved and interpolation, the second inequality in \eqref{eq:uniform_bounddness} follows from the first.

In order to establish the first inequality in \eqref{eq:uniform_bounddness}, we write $b=b^\flat+b^\sharp$ where
\begin{equation}\label{decomposition of bj}
    b^\flat(k,\xi)=b(k,\xi)(1-\chi_0(\xi)),\quad \mbox{and}\quad b^\sharp(k,\xi)=b(k,\xi)\chi_0(\xi).
\end{equation}
Now since $m\leq 0$ and $1-\chi_0$ is a low frequency cut-off, one can throw away the $\omegaF$ in the definition of $b$ which would then eliminate the $k$-dependency in $b^\flat$. Then by the kernel estimates for the FIOs with amplitude $b^\flat$ (see e.g. Lemma \ref{main low frequency estim}), for $f\in \bmo$ we have that 
\[
        \norm{P_k^{u} T_{b^\flat}^{\phase}(f)}_{L^\infty}\lesssim \norm{T_{b^\flat}^{\phase}(f)}_{L^\infty}\lesssim \norm{(1-\chi_0)(D)f}_{L^\infty}\lesssim \norm{f}_{\bmo}.
\]

 In order to ameliorate $(P_{k}^{u} \circ T^{\phase}_{b^\sharp})(f)$ so that we can better understand its action on $\bmo$ functions, we employ an argument from \cite{Monster}*{page 27}.
According to that argument, for $n\geq 2$ and $m=\frac{-(n-1)}{2}$, one introduces an operator 
\begin{equation}\label{defn: Rk}
   R_k= \sum_{j=k_0}^{k} Q_j 2^{(k-j)m}
\end{equation}
 with some positive $k_0$,  which enables one to replace $(P_{k}^{u} \circ T^{\phase}_{b^\sharp})(f)$ by $P_k^{u} \circ R_k \circ T^{\phase}_{\gamma}(f)$, for $n\geq 2$, where $\gamma(\xi):= \chi_0(\xi)|\xi|^m \in S^{-(n-1)/2}$.
 
 By Lemma 4.8 in \cite{Monster}, the operator $R_k$ has a kernel $K_k$ which has the following properties:
\begin{equation*}
	\int K_k(z) \dd z=0;
\end{equation*}
and for each $0<\delta<\frac{n-1}{2}$ the estimates
\begin{equation*}
	\abs{K_k(x-y)}\lesssim {2^{kn}}{\brkt{1+\frac{\abs{x-y}}{2^{-k}}}^{-n-\delta}}
\end{equation*}
and
\begin{equation*}
	\abs{K_k(x-y)-K_k(x-y')}\lesssim {2^{k(n+1)}}\abs{y-y'}
\end{equation*}
hold for all $x,y,y'\in \R^n$ and $k \in \Z$. Therefore the operator $R_k$ satisfies
\begin{equation*}
	\sup_{k\in\Z}\norm{R_k f}_{L^q}\lesssim \norm{f}_{L^q},\qquad 1\leq q< \infty,
\end{equation*}
and
\begin{equation*}\label{t_eq:BMO_bound}
\sup_{k\in\Z}\norm{R_k f}_{L^\infty}\lesssim \norm{f}_{\mathrm{BMO}}.
\end{equation*}
Using this $\mathrm{BMO}$--$L^\infty$ boundedness (valid for $n\geq 2$), the global $\bmo$-boundedness of FIOs with amplitudes in $S^{-(n-1)/{2}}$ (i.e.\ Theorem \ref{linearhpthm}) and the $L^\infty$-boundedness of $P_k^u$ yield that 
\[
        \sup_{k}\norm{P_k^{u} T_{b^\sharp}^{\phase}(f)}_{L^\infty}\lesssim \norm{\chi_0(D)f}_{\mathrm{BMO}}\leq \norm{f}_{\bmo}.
\]
\end{proof}

\begin{rem}
Here we see that the assumption $n\geq2$ is used in the proof of \emph{Lemma \ref{lem:Wase-Lin}} to ensure that $\delta$ can be chosen positive. This is not just a feature of the proof and is in fact necessary. As was shown in \emph{Proposition 5.3 in \cite{Monster}}, the bilinear operator in dimension $n=1$ with amplitude $\sigma \equiv 1$ and phase functions $\phase_1 = x\xi +|\xi|$, $\phase_2 = x\eta$ and $\phase_3 =0$ fails to be bounded from $\bmo \times L^2$ to $L^2$.
\end{rem}

\subsubsection{$L^2 \times \bmo \times\dots\times \bmo \to L^2$}

Here we take  $m_1=0$ and $f_1 \in L^2$ and $m_j= -(n-1)/2$, $f_j \in \bmo$ 
for $j=2,\ldots, N$. Noting that $b_1$ does not depend on $k$, the quadratic estimate
\m{
\left(\sum_{k\equiv \ell\Mod{3},\,k\geq k_0 } \int |(Q_k^{u_1}\circ T^{\phase_1}_{b_1})(f_1)(x)|^2 \, \dd x \right)^{1/2} \lesssim \norm{T^{\phase_1}_{b_1}(f_1)}_{L^2}\lesssim \norm{f_1}_{L^2}
}
follows with the help of \eqref{ineq:quadraticestimate} and Theorem \ref{linearhpthm}.
Applying this and \eqref{eq:uniform_bounddness} to the expression  \eqref{t_l2targetsecondfactor} yields
\[
\begin{split}
    &\left(\sum_{k\equiv \ell\Mod{3},\,k\geq k_0 } \int \abs{(Q_k^{u_1}\circ T^{\phase_1}_{b_1})(f_1)(x)\prod_{j=2}^N(P_{k}^{u_j} \circ T^{\phase_j}_{b_j})(f_j)}^2\, \dd x\right)^{1/2}\\
    &\lesssim \left(\sum_{k\equiv \ell\Mod{3},\,k\geq k_0 } \int |(Q_k^{u_1}\circ T^{\phase_1}_{b_1})(f_1)(x)|^2\, \dd x\right)^{1/2} \prod_{j=2}^N\sup_k\norm{(P_{k}^{u_j} \circ T^{\phase_j}_{b_j})(f_j)}_{L^\infty} \\
    &\leq \norm{f_1}_{L^2}\prod_{j=2}^N\norm{f_j}_{\bmo}.
\end{split}
\]

\subsection{Boundedness with the target space $h^1$.}

Now we take $m_0 = -(n-1)/2$ and so by duality and \eqref{the magic operator} we have
\nma{testing_with_bmo}{
& \norm{B_\ell(f_1,\dots,f_N)}_{h^1}  = \\
&\sup_{\|f_0\|_{\bmo}=1} \left|\sum_{k\equiv \ell\Mod{3},\,k\geq k_0 } \int  \,Q_{k}^0( \chi_0 (2D)f_0)\, (Q_{k}^{u_1} \circ T^{\phase_1}_{b_1})(f_1)\prod_{j=2}^N(P_{k}^{u_j} \circ T^{\phase_j}_{b_j})(f_j) \, \dd x\right|.
}
Since $f_0\in \bmo$, we have that $\chi_0(2D)f_0 \in \mathrm{BMO}$. Therefore
\m{
    \dd\mu_{f_0}(x,t) := \sum_{k\equiv \ell\Mod{3},\,k\geq k_0 } |Q_{k}^0( \chi_0(2D)f_0)(x)|^2 \dd x \,\delta_{2^{-k}}(t)
}
is a Carleson measure with Carleson norm not exceeding a constant multiple of $\Vert f_0 \Vert_{\bmo}^2$.

\subsubsection{$\bmo\times \dots\times\bmo\times h^1 \to h^1$}\label{sec:bmoh1h1}

Here we take $m_j = -(n-1)/2$ for $j=0, \dots\, N$, $f_j\in \bmo$, $j=0,\dots, N-1$ and $f_N\in h^1$ in \eqref{testing_with_bmo}.

Since $f_1\in \bmo$, Theorem \ref{linearhpthm} and \eqref{defn:bmo} yield that $T^{\phase_1}_{b_1}(f_1)\in \mathrm{BMO}$ and therefore
\m{
\dd \mu_{f_1}(x,t) := \sum_{k\equiv \ell\Mod{3},\,k\geq k_0 } |Q_{k}^u \circ T^{\phase_1}_{b_1}(f_1)(x)|^2 \, \dd x\, \delta_{2^{-k}}(t)
}
is a Carleson measure. Since we also have that
\ma{
&\left|Q_{k}^0 \circ \chi_0 (2D)(f_0)(x)\right|\left|Q_{k}^{u_1} \circ T^{\phase_1}_{b_1}(f_1)(x)\right| \\
&\leq \frac{1}{2} \left(\frac{\norm{f_1}_\bmo}{\norm{f_0}_\bmo}\left|Q_{k}^0 \circ \chi_0 (2D)(f_0)(x)\right|^2 +  \frac{\norm{f_0}_\bmo}{\norm{f_1}_\bmo}\left|Q_{k}^{u_1} \circ T^{\phase_1}_{b_1}(f_1)(x)\right|^2\right)
}
the measure
\m{
    \dd \mu_{f_0,f_1}(x,t):=\sum_{k\equiv \ell\Mod{3},\,k\geq k_0 } Q_{k}^0 \circ \chi_0 (2D)(f_0)(x) Q_{k}^{u_1} \circ T^{\phase_1}_{b_1}(f_1)(x) \dd x \delta_{2^{-k}}(t)
}
is also Carleson with Carleson norm bounded by $\norm{f_0}_\bmo\norm{f_1}_\bmo$. Moreover, by \eqref{eq:uniform_bounddness}, even
\nm{eq:carlon}{
    \dd\mu_{f_0\dots, f_{N-1}}(x,t) :=\sum_{k} \prod_{j=2}^{N-1}(P_{k}^{u_j} \circ T^{\phase_j}_{b_j})(f_j)(x)\, \delta_{2^{-k}}(t)\dd\mu_{f_0,f_1}(x,t)
}
is a Carleson measure.

At this point we repeat the decomposition \eqref{decomposition of bj} of $b_N$ into the sum $b^\flat_N+b^\sharp_N$.\footnote{This is necessary because $b_N$ depends on $k$. Had it not done so, the proof of this endpoint could be completed by arguing as in \eqref{est:decomcandd} directly with $b_N$ instead of $b^\flat_N$.} We can see that since $m_N=-(n-1)/2$ and $1-\chi_0$ is a low frequency cut-off, one can disregard the $\omegaF_k$ in the definition of $b^\flat_N$ which means $b^\flat_N$ is independency of $k$. Then the characterisation \ref{eq:hp} of local Hardy spaces in Definition \ref{def:Triebel} and \eqref{carleson estim} yields
\nma{est:decomcandd}{
 &\left|\sum_{k\equiv \ell\Mod{3},\,k\geq k_0 } \int (P_{k}^{u_N} \circ T^{\phase_N}_{b^\flat_N})(f_N)\, \dd\mu_{f_0,\dots, f_{N-1}}(x,2^{-k})\right|\\
 &\lesssim \prod_{j=0}^{N-1}\norm{f_j}_\bmo \int_{\R^n}\sup_{k\geq k_0} \sup_{|x-y|<2^{-k}} |(P_{k}^{u_N} \circ T^{\phase_N}_{b^\flat_N})(f_N)|\, \dd x \\
&\lesssim  \Vert  T^{\phase_N}_{b^\flat_N}(f_N)\Vert_{h^1} \prod_{j=0}^{N-1}\norm{f_j}_\bmo \lesssim \prod_{j=0}^{N-1} \norm{f_j}_\bmo \Vert f_N \Vert_{h^1}.
}

To deal with $(P_{k}^{u_N} \circ T^{\phase_N}_{b^\sharp_N})(f_N)$ we continue to follow the proof of Lemma \ref{lem:Wase-Lin} and replace it by $P_k^{u_N} \circ R^{N}_k \circ T^{\phase_N}_{\gamma}(f_N)$, where $\gamma\in S^{m_N}$. Lemma \ref{lem:carlesoncomp} leads us to conclude that
\m{
\sum_k R^{N*}_{k}(\dd\mu_{f_0,\dots, f_{N-1}}(\cdot,2^{-k}))(x) \, \delta_{2^{-k}}(t)\dd x
}
is also a Carleson measure. This via \eqref{ineq:carlh1} yields
\ma{
 &\left|\sum_{k\equiv \ell\Mod{3},\,k\geq k_0 } \int (P_{k}^{u_N} \circ T^{\phase_N}_{b^\sharp_N})(f_N)\, \dd \mu_{f_0,\dots, f_{N-1}}(x,2^{-k})\right|\\
  &= \left|\sum_{k\equiv \ell\Mod{3},\,k\geq k_0 } \int (P_{k}^{u_N} \circ T^{\phase_N}_{\gamma})(f_N)\, R^{N*}_{k}(\dd \mu_{f_0,\dots, f_{N-1}}(\cdot,2^{-k})) \, \dd x\right|\\
 &\lesssim  \Vert  T^{\phase_N}_{\gamma}(f_N)\Vert_{h^1} \prod_{j=0}^{N-1}\norm{f_j}_\bmo \lesssim \prod_{j=0}^{N-1} \norm{f_j}_\bmo \Vert f_N \Vert_{h^1}.
}

\subsubsection{$h^1 \times \bmo\times \cdots\times\bmo \to h^1$}

Here we take $m_j = -(n-1)/2$, for $j=0,\dots, \, N$, $f_0\in \bmo$, $f_1\in h^1$ and $f_j\in\bmo$, $j=2,\,\dots,\,N$. 
Using \eqref{eq:uniform_bounddness} from Lemma \ref{lem:Wase-Lin}, we have 
\[
    \sup_{k}\norm{P_k^{u_j} T_{b_j}^{\phase_j}(f_j)}_{L^\infty}\lesssim  \norm{f_j}_{\bmo}, \qquad j=2,\ldots, N.
\]
We now take 
\begin{equation*}
    \begin{split}
         G(x) &= \chi_0(2D)(f_0)(x),\qquad 
        v(2^{-k},x) = \prod_{j=2}^{\infty}P_k^{u_j} \circ T^{\phase_j}_{b_j}(f_j), \qquad \mbox{and}\\
        F(x) &= (Q_{k}^{u_1} \circ T^{\phase_1}_{b_1})(f_1) = (Q_{k}^{u_1} \circ T^{\phase_1}_{b_1} \circ \chi_0 (2D))(f_1),
    \end{split}
\end{equation*} 
and thereafter apply Proposition \ref{cor:monster 1}, Theorem \ref{linearhpthm} and \eqref{defn:bmo} to the right-hand side of \eqref{testing_with_bmo} to obtain
\ma{
 \sup_{\|f_0\|_{\bmo}=1} & \left|\sum_{k\equiv \ell\Mod{3},\,k\geq k_0 } \int  \,Q_{k}^0( \chi_0 (2D)f_0)\, (Q_{k}^{u_1} \circ T^{\phase_1}_{b_1})(\chi_0 (2D)f_1)\prod_{j=2}^N(P_{k}^{u_j} \circ T^{\phase_j}_{b_j})(f_j) \, \dd x\right| \\
&\lesssim  \Vert      \chi_0 (2D)f_1\Vert_{H^1}\prod_{j=2}^{N} \Vert f_j\Vert_{\bmo}\lesssim \Vert      f_1\Vert_{h^1}\prod_{j=2}^{N} \Vert f_j\Vert_{\bmo},
}
{{where we have also used \eqref{eq:local_H1} in dealing with $\Vert \chi_0 (2D)f_1\Vert_{H^1}.$}}
\subsubsection{$L^2\times L^2 \times \bmo\times \ldots\times \bmo\to h^1$}

We choose $m_1=m_2=0$, $f_1,f_2\in L^2$, $m_j=-\frac{n-1}{2}$ for $j=0$ and $j=3,\ldots N$ and $f_0\in \bmo$. Starting once again with \eqref{testing_with_bmo}, we have that for all $\|f_0\|_{\bmo}=1$
\nma{ineq:l2l2h1}{
 & \left|\sum_{k\equiv \ell\Mod{3},\,k\geq k_0 } \int  \,Q_{k}^0( \chi_0 (2D)f_0)\, (Q_{k}^{u_1} \circ T^{\phase_1}_{b_1})(f_1)\prod_{j=2}^N(P_{k}^{u_j} \circ T^{\phase_j}_{b_j})(f_j) \, \dd x\right| \\
&\lesssim  \brkt{\int \sum_{k\geq k_0} \abs{Q_k^{u_1}\circ T_{b_1}^{\phase_1}(f_1)}^2\dd x}^{1/2} \\
&\left(\sum_{k\equiv \ell\Mod{3},\,k\geq k_0 } \int  \abs{P_{k}^{u_2} \circ T^{\phase_2}_{b_2}(f_2) }^2 \abs{Q_{k}^0( \chi_0 (2D)f_0)}^2\, \prod_{j=3}^N \abs{P_{k}^{u_j} \circ T^{\phase_j}_{b_j}(f_j) }^2\, \dd x\right)^{1/2}
}
Since $f_j\in \bmo$ for $j=0,3,\ldots, N$, we can argue as we did for \eqref{eq:carlon} to conclude
\[
    \sum_{k\geq k_0} \abs{Q_{k}^0( \chi_0 (2D)f_0)}^2\, \prod_{j=3}^N \abs{P_{k}^{u_j} \circ T^{\phase_j}_{b_j}(f_j) }^2\dd x\, \delta_{2^{-k}}(t)
\]
defines a Carleson measure with Carleson norm bounded by $\norm{f_0}_{\bmo}^2\prod_{j=3,\ldots N }\norm{f_j}_{\bmo}^2$. The $L^2$-boundedness of FIOs from Theorem \ref{linearhpthm}, together with a quadratic estimate \eqref{ineq:quadraticestimate} in the first factor and a non-tangential maximal function estimate \eqref{ineq:carll2} in the second, yield that \eqref{ineq:l2l2h1} is bounded by 
\begin{equation*}
    \norm{f_1}_{L^2}\times \norm{f_2}_{L^2}\prod_{j=3}^N \norm{f_j}_{\bmo}.  
\end{equation*}
We would also like to note that when $N=2$ then the functions $f_j$, with $j=3, \dots, N$ do not appear in the estimates above.
\subsubsection{$\bmo \times L^2 \times L^2 \times \bmo\times \ldots\times \bmo\to h^1$}

We choose $m_2=m_3=0$, $f_2,f_3\in L^2$, $m_j=-\frac{n-1}{2}$ for $j=0,1,4,\ldots N,$ and $f_0, \, f_1$ are both in $\bmo$. Continuing from \eqref{testing_with_bmo}, we have that for all $\|f_0\|_{\bmo}=1$
\ma{
 & \left|\sum_{k\equiv \ell\Mod{3},\,k\geq k_0 } \int  \,Q_{k}^0( \chi_0 (2D)f_0)\, (Q_{k}^{u_1} \circ T^{\phase_1}_{b_1})(f_1)\prod_{j=2}^N(P_{k}^{u_j} \circ T^{\phase_j}_{b_j})(f_j) \, \dd x\right| \\
&\lesssim  \brkt{\int \sum_{k\geq k_0} \abs{P_k^{u_3} T_{b_3}^{\phase_3}(f_3)}^2\abs{Q_k^{u_1}\circ T_{b_1}^{\phase_1}(f_1)}^2\dd x}^{1/2} \times \\
&\left(\sum_{k\equiv \ell\Mod{3},\,k\geq k_0 } \int  \abs{P_{k}^{u_2} \circ T^{\phase_2}_{b_2}(f_2) }^2 \abs{Q_{k}^0( \chi_0 (2D)f_0)}^2\, \prod_{j=4}^N \abs{P_{k}^{u_j} \circ T^{\phase_j}_{b_j}(f_j) }^2\, \dd x\right)^{1/2}
}
Since $f_j\in \bmo$ for $j=0$ and $j=4,\ldots, N$, arguing once again as we did for \eqref{eq:carlon}, we see
\[
    \sum_{k\geq k_0} \abs{Q_{k}^0( \chi_0 (2D)f_0)}^2\, \prod_{j=4}^N \abs{P_{k}^{u_j} \circ T^{\phase_j}_{b_j}(f_j) }^2\dd x\, \delta_{2^{-k}}(t),
\]
is a Carleson measure with Carleson norm bounded by $\norm{f_0}_{\bmo}^2\prod_{j=4}^N\norm{f_j}_{\bmo}^2$, and similarly 
\[
    \sum_{k\geq k_0} \abs{Q_k^{u_1}\circ T_{b_1}^{\phase_1}(f_1)}^2 \dd x\, \delta_{2^{-k}}(t)
\]
defines a Carleson measure with Carleson norm bounded by $\norm{f_1}_{\bmo}^2$. The $L^2$ boundedness of FIOs (Theorem \ref{linearhpthm} and the non-tangential maximal function estimate \eqref{ineq:carll2} yields that the right-hand side of the inequality above is bounded by 
\[
    \norm{f_2}_{L^2}\times \norm{f_3}_{L^2}\prod_{j=4}^N \norm{f_j}_{\bmo}.  
\]

\subsection{Boundedness with the target space $\bmo$.}

Here the only case to consider is the $\bmo\times\dots\times\bmo \to \bmo$ boundedness of the operator in \eqref{the magic operator}. In this case we take $m_j = -(n-1)/2$, $j=0,\dots\, N$, $f_0\in h^1$ and $f_j\in \bmo$ for $j=1,\dots, N$. Using \eqref{the magic operator} and duality, pairing against  $f_0$, we must bound 

\nma{testing_with_h1}{
\sup_{\|f_0\|_{h^1}=1} \left|\sum_{k\equiv \ell\Mod{3},\,k\geq k_0 } \int {Q_{k}^0(\chi_0(2D)f_0)(x)}\, (Q_{k}^{u_1} \circ T^{\phase_1}_{b_1})(f)(x)\prod_{j=2}^N(P_{k}^{u_j} \circ T^{\phase_j}_{b_j})(f_j) \, \dd x\right|.
}

To bound this further we apply Proposition \ref{cor:monster 1}. We take $F(x) = \chi_0 (2D)f_0(x)$, $G(x) = T^{\phase_1}_{b_1}(f_1)(x)$ and $v(2^{-k},x) = \prod_{j=2}^N(P_{k}^{u_j} \circ T^{\phase_j}_{b_j})(f_j)(x)$. 
Clearly  $\norm{F}_{H^1} \lesssim \norm{f_0}_{h^1}=1$, and \eqref{defn:bmo} and Theorem \ref{linearhpthm} yield that $\norm{G}_{\BMO} \lesssim \norm{f_1}_{\bmo}$ and applying \eqref{eq:uniform_bounddness} from Lemma \ref{lem:Wase-Lin}, we have that $$\norm{v(2^{-k},t)}_{L^{\infty}_{k,x}} \leq \prod_{j=2}^N \norm{P_{k}^{u_j} \circ T^{\phase_j}_{b_j}(f_j)(x)}_{L^{\infty}_{k,x}} \lesssim \prod_{j=2}^N\norm{f_j}_\bmo.$$ 
It follows that \eqref{testing_with_h1} is bounded by $\prod_{j=1}^N\norm{f_j}_\bmo$, as required.

\section{Boundedness results for $T^\Phase_{\sigma_{j,k}}$.}\label{sec:sigmaij}

Our analysis of $T^\Phase_{\sigma_{j,k}}$ begins very similarly to that of $T^\Phase_{\sigma_{j}}$ in Section \ref{sec:sigmaj}. Just as in that case, the symmetry of the operators form under permutations of the frequency variables allows us to restrict our attention to just one of the $\sigma_{j,k}$, the argument for all the others being identical. We choose to study $\sigma_{1,2}$, so we have that $|\xi_1|$ and $|\xi_2|$ are comparable to each other. More precisely, we know that
\m{
c_1c_3\abs{\Xi}^2 \leq \abs{\xi_1}^2 \quad \mbox{and} \quad c_1c_3\abs{\Xi}^2 \leq \abs{\xi_2}^2 \quad \mbox{so} \quad c_1c_3\abs{\xi_1}^2 \leq \abs{\xi_2}^2 \leq \frac{1}{c_1c_3} \abs{\xi_1}^2
}
on the $\Xi$-support of $\sigma_{1,2}(x,\Xi)$, with the constants $c_1$ and $c_3$ being the same as those in Section~\ref{sec:freqdecom}.  We choose an integer $k_1$ so that $2^{-k_1} \leq c_1c_3$ and define $\widehat{\zeta}_k$ via
\begin{itemize}

\item $\widehat{\zeta}_k(\xi)^2 := \vth(2^{-k-k_1-2}\xi)^2 - \vth(2^{3+k_1-k}\xi)^2$,

\end{itemize}
so that when $\psiF_k(\xi_1) \neq 0$ and $\sigma_1(x,\Xi) \neq 0$, then
\ma{
2^{-k_1-2}  &\leq c_1c_3|2^{-k}\xi_1| \leq |2^{-k}\xi_2| \leq \frac{1}{c_1c_3}|2^{-k}\xi_1| \leq 2^{k_1+2} \\
&\quad \mbox{which implies} \quad \widehat{\zeta}_k(\xi_2) = 1.
}
With the same choice of $\psiF_k$, $\th_k$ and $\chi_0$ from Section \ref{sec:sigmaj}, we can argue as we did there to write $T^\Phase_{\sigma_{1,2}}$ as
\nma{decompsigma111}{
&T^\Phi_{\sigma_{1,2}}(f_1,\ldots,f_N)(x) \\
&= \int_{\R^{nN}} \sum_{k\geq k_0} \psiF_k(\xi_1)^2\widehat{\zeta}_k(\xi_2)^2   \sigma_{1,2}(x,\Xi)
\chi_0(\xi_1)\widehat{f_1}(\xi_1) \times \\
& \qquad \chi_0(\xi_2)\widehat{f_2}(\xi_2)\prod_{j=3}^N\th_k(\xi_j)^2\widehat{f_j}(\xi_j) e^{ix\cdot(\xi_1 +\dots+\xi_N)+i\Phase(\Xi)} \ddd\Xi,
}
and then define
\m{
{\mathfrak a}(k,x,\Xi) = \sigma_{1,2}(x,2^{k}\Xi)\psiF_1(\xi_1)\psiF_1(\xi_2)\prod_{j=2}^N\th_1(\xi_j)\left(\frac{2^{-k}}{|\xi_1|}\right)^{m-m_2}\left(\frac{2^{-k}}{|\xi_2|}\right)^{m_2},
}
where once again $\sum_{j=1}^N m_j = m$, so that using the notation \eqref{the complete phase} the expression \eqref{decompsigma111} can be rewritten as
\nma{decompsigma112}{
& \sum_{k\geq k_0}\int_{\R^{nN}}  {\mathfrak a}(k,x,2^{-k}\Xi)2^{km_0}\widehat{\theta_k}(\xi_1+\ldots+\xi_N) \abs{2^{-k}\xi_1}^{m - m_2 - m_1}\psiF_k(\xi_1) \times \\
&\qquad \abs{\xi_1}^{m_1}\chi_0(\xi_1)\widehat{f_1}(\xi_1)\widehat{\zeta}_k(\xi_2)\abs{\xi_2}^{m_2}\chi_0(\xi_2)\widehat{f_2}(\xi_2)\prod_{j=3}^N2^{km_j}\th_k(\xi_j)^2\widehat{f_j}(\xi_j) e^{i\Phase(x,\Xi)} \ddd\Xi.
}
Just as in Section \ref{sec:sigmaj}, the Fourier inversion formula yields
\m{
{\mathfrak a}(k,x,\Xi) = \int \frac{{\mathfrak m}(k,x,U)}{(1+|U|^2)^M}\, e^{i\,\Xi\cdot U}\, \, \ddd U,\qquad U=(u_1,\ldots,u_N),
}
for a smooth bounded function ${\mathfrak m}$.
So \eqref{decompsigma112} can be written as a weighted average in $U =(u_1,\dots,u_N)$ of
\begin{equation*}
\begin{split}
&\sum_{k=k_0}^\infty {\mathfrak m}(k,x,U)\int 2^{km_0}\th_k(\xi_1+\ldots+\xi_N)   \left[\abs{2^{-k}\xi_1}^{m-m_1-m_2}\psiF_k(\xi_1)e^{i2^{-k}\xi_1\cdot u_1}\abs{\xi_1}^{m_1}\chi_0 (\xi_1)\widehat{f_1}(\xi_1)\right]\\
&\qquad
   \left[\widehat{\zeta}_k(\xi_2)e^{i2^{-k}\xi_2\cdot u_2}\abs{\xi_2}^{m_2}\chi_0 (\xi_2)\widehat{f_2}(\xi_2)\right] \left[\prod_{j=3}^N2^{km_j}\th_k(\xi_j) \widehat{f_j}(\xi_j)e^{i2^{-k} \xi_j\cdot u_j}\right]  e^{i\Phase(x,\Xi)}  \ddd\Xi.
\end{split}
\end{equation*}
Therefore we need to prove the boundedness of the following operator
\nma{t_piecedecomp for sigma N+1}{
 & D(f_1,\ldots,f_N)(x) \\
 &=\sum_{k=k_0}^\infty  M_{\mathfrak m}\circ T^{\phase_0}_{d_0} \circ P_{k}^0\left[(Q_{k}^{u_1} \circ T^{\phase_1}_{d_1})(f_1)\, (Q_{k}^{u_2} \circ T^{\phase_2}_{d_2})(f_2)\,\prod_{j=3}^N(P_{k}^{u_j} \circ T^{\phase_j}_{d_j})(f_j)\right](x),
}
where
\begin{align*}
\widehat{P_k^0(f)}(\xi) &= \th_k(\xi)\widehat{f}(\xi), & d_0(k,\xi) &= 2^{km_0}\omegaF_k(\xi), \\
\widehat{Q_k^{u_1}(f)}(\xi) &= \abs{2^{-k}\xi}^{m-m_1-m_2}\psiF_k(\xi)e^{i2^{-k} \xi\cdot u_1}\widehat{f}(\xi), & d_1(\xi) &= |\xi|^{m_1}\chi_0 (\xi), \\
\widehat{Q_k^{u_2}(f)}(\xi) &= \widehat{\zeta}_k(\xi)e^{i2^{-k} \xi\cdot u_2}\widehat{f}(\xi), & d_2(\xi) &= |\xi|^{m_2}\chi_0 (\xi), \\
\widehat{P_k^{u_j}(f)}(\xi) &= \th_k(\xi)e^{i2^{-k} \xi\cdot u_j}\widehat{f}(\xi), & d_j(k,\xi) &= 2^{km_j}\omegaF_k(\xi),
\end{align*}
for $j=3,\ldots,N$, $\omegaF_k(\xi) := \th_k(\xi/2)$ is a bump function equal to one on the support of $\th_k$, and $M_{\mathfrak m}$ denotes multiplication by ${\mathfrak m}$.\footnote{The same ambiguity of notation arises here as in \eqref{t_piecedecomp}. See footnote \ref{abiguity}.}

We now proceed to consider all the necessary endpoint cases. Just as in Section~\ref{sec:sigmaj}, due to the symmetry of the form of \eqref{t_piecedecomp for sigma N+1} in the indicies $j=1,2$ and $j=3,\dots,N$ we only need to consider endpoint cases $(p_0,\dots,p_N)$ which are distinct within the equivalence class of permutations of $(p_1,p_2)$ and $(p_3,\dots,p_N)$. In each case we fix 
\[
    \frac{1}{p_0}=\sum_{j=1}^{N }\frac{1}{p_j},\qquad 1\leq p_j\leq \infty,\quad j=0,\ldots,N,
\]
and 
\[
    m_j:=-(n-1)\abs{\frac{1}{2}-\frac{1}{p_j}},\qquad j=0,\ldots, N,
\]
and consider $f_j \in X^{p_j}$ for $j=1,\dots,N$. Using duality in \eqref{t_piecedecomp for sigma N+1} it is enough to estimate 
\begin{equation}\label{eq:Lu-Be}
\sum_{k\geq k_0}
\int P_{k}^0\circ T^{-\phase_0}_{d_0} (M_{\mathfrak m} f_0)\left[(Q_{k}^{u_1} \circ T^{\phase_1}_{d_1})(f_1)\, (Q_{k}^{u_2} \circ T^{\phase_2}_{d_2})(f_2)\,\prod_{j=3}^N(P_{k}^{u_j} \circ T^{\phase_j}_{d_j})(f_j)\right]\dd x,  
\end{equation}
for $f_0 \in X^{p_0'}$ with $\norm{f_0}_{X^{p_0'}} = 1$.

Comparing this analysis with that of $T_{\sigma_j}^\Phase$, observe that what was $Q_{k}^0$ (a multiplier supported on an annulus) in \eqref{t_piecedecomp} has been replaced by $P_{k}^0$ (a multiplier supported on a ball) in \eqref{t_piecedecomp for sigma N+1}. This means that our technique to remove the dependency of $M_{\mathfrak m}$ on $k$ will no longer be directly applicable. In the case $p_0 = 1$ and $p_0 = \infty$, the $k$ dependency is not problematic, and methods already introduced in Section \ref{sec:sigmaj} can be successfully applied again here. In the case $p_0 = 2$ this dependency is more problematic. The possibility of replacing $(P_{k}^{u} \circ T^{\phase}_{d})(f)$ with $P_k^{u} \circ R_k \circ T^{\phase}_{\gamma}(f)$ as in Lemma \ref{lem:Wase-Lin} is not available to us, since $m_0=0$, and therefore this method does not allow us to use \eqref{ineq:quadraticestimate} to estimate the $f_0$ term. We present an alternative approach which can successfully deal with this $k$-dependency in this case in Section \ref{sec:carlsmallnorm}.

\subsection{The endpoint cases with target space $L^2$}\label{sec:carlsmallnorm}

We write
\nm{eq:lpwithl2target}{
P^0_k = \sum_{\ell=k_0+1}^k Q_\ell + P^0_{k_0}
}
where $\widehat{Q_\ell(f)}(\xi) = \brkt{\th_\ell(\xi) - \th_{\ell-1}(\xi)}\widehat{f}(\xi)$ and so \eqref{eq:Lu-Be} is the sum of
\nma{eq:lowfreq11}{
&\sum_{k\geq k_0}
\int P_{k_0}^0\circ T^{-\phase_0}_{d_0} (M_{\mathfrak m} f_0)\left[(Q_{k}^{u_1} \circ T^{\phase_1}_{d_1})(f_1)\, (Q_{k}^{u_2} \circ T^{\phase_2}_{d_2})(f_2)\,\prod_{j=3}^N(P_{k}^{u_j} \circ T^{\phase_j}_{d_j})(f_j)\right]\dd x \\
&\lesssim \brkt{\sum_{k\geq k_0}
\int \abs{P_{k_0}^0\circ T^{-\phase_0}_{d_0} (M_{\mathfrak m} f_0) \, (Q_{k}^{u_2} \circ T^{\phase_2}_{d_2})(f_2)}^2 \dd x}^{1/2} \times \\
&\qquad \brkt{\sum_{k\geq k_0}
\int \abs{(Q_{k}^{u_1} \circ T^{\phase_1}_{d_1})(f_1)\,\prod_{j=3}^N(P_{k}^{u_j} \circ T^{\phase_j}_{d_j})(f_j)}^2\dd x}^{1/2}
}
and
\begin{align*}
&\sum_{k=k_0}^\infty\sum_{\ell=k_0}^k
\int Q_\ell\circ T^{-\phase_0}_{d_0} (M_{\mathfrak m} f_0)\left[(Q_{k}^{u_1} \circ T^{\phase_1}_{d_1})(f_1)\, (Q_{k}^{u_1} \circ T^{\phase_2}_{d_2})(f_2)\,\prod_{j=3}^N(P_{k}^{u_j} \circ T^{\phase_j}_{d_j})(f_j)\right]\dd x \\
&= \sum_{\ell=k_0}^\infty\sum_{k=\ell}^\infty
\int Q_\ell\circ T^{-\phase_0}_{d_0} (M_{\mathfrak m} f_0)\left[(Q_{k}^{u_1} \circ T^{\phase_1}_{d_1})(f_1)\, (Q_{k}^{u_1} \circ T^{\phase_2}_{d_2})(f_2)\,\prod_{j=3}^N(P_{k}^{u_j} \circ T^{\phase_j}_{d_j})(f_j)\right]\dd x \\
&= \sum_{\ell=k_0}^\infty\sum_{k=0}^\infty
\int Q_\ell\circ T^{-\phase_0}_{d_0} (M_{\mathfrak m} f_0)\left[(Q_{\ell+k}^{u_1} \circ T^{\phase_1}_{d_1})(f_1)\, (Q_{\ell+k}^{u_1} \circ T^{\phase_2}_{d_2})(f_2)\,\prod_{j=3}^N(P_{\ell+k}^{u_j} \circ T^{\phase_j}_{d_j})(f_j)\right]\dd x,
\end{align*}
where we remind the reader that now in the last expression $d_0$ and $M_{\mathfrak m}$ depend on $\ell+k$ and we have taken $m_0 = 0$. Given the frequency support properties on $Q_\ell\circ T^{-\phase_0}_{d_0}$, we can redefine $d_0\equiv1$ without changing the operator and so make it independent of $k+\ell$. Equally, the composition $Q_\ell\circ T^{-\phase_0}_{d_0} \circ M_{\mathfrak m}$ can be replaced by a finite sum of operators of the form $Q_\ell\circ T^{-\phase_0}_{d_0} \circ M_{k}$ where $M_k$ only depends on $k$ (and $x$ and $U$), in the same way as we obtained \eqref{the magic operator}. Thus our task is to bound 
\nma{highfreq11}{
&\sum_{\ell=k_0}^\infty \int Q_\ell\circ T^{-\phase_0}_{d_0} (M_k f_0) \, (Q_{k+\ell}^{u_1} \circ T^{\phase_1}_{d_1})(f_1)\, (Q_{k+\ell}^{u_1} \circ T^{\phase_2}_{d_2})(f_2) \,\prod_{j=3}^N(P_{k+\ell}^{u_j} \circ T^{\phase_j}_{d_j})(f_j)\dd x \\
&\lesssim \brkt{\sum_{\ell\geq k_0}
\int \abs{Q_\ell\circ T^{-\phase_0}_{d_0} (M_k f_0)}^2 \dd x}^{1/2} \times \\
&\qquad \brkt{\sum_{\ell\geq k_0}
\int \abs{(Q_{k+\ell}^{u_1} \circ T^{\phase_1}_{d_1})(f_1)\,(Q_{k+\ell}^{u_1} \circ T^{\phase_2}_{d_2})(f_2)\,\prod_{j=3}^N(P_{k+\ell}^{u_j} \circ T^{\phase_j}_{d_j})(f_j)}^2\dd x}^{1/2} \\
&\lesssim \norm{f_0}_{L^2}\brkt{\sum_{\ell\geq k_0}
\int \abs{(Q_{k+\ell}^{u_1} \circ T^{\phase_1}_{d_1})(f_1)\,(Q_{k+\ell}^{u_1} \circ T^{\phase_2}_{d_2})(f_2)\,\prod_{j=3}^N(P_{k+\ell}^{u_j} \circ T^{\phase_j}_{d_j})(f_j)}^2\dd x}^{1/2}
}
(where we made use of \eqref{ineq:quadraticestimate}) so that it is summable in $k$, plus we must, of course, bound \eqref{eq:lowfreq11}.

We begin by further estimating the first factor on the right-hand side of \eqref{eq:lowfreq11}. In each endpoint case below we will have $p_2=\infty$ so that
\begin{equation}
     \sum_{k\geq k_0} \abs{(Q_{k}^{u_2} \circ T^{\phase_1}_{d_2})(f_2)(x)}^2\dd x\delta_{2^{-k}}(t) 
\end{equation}
is always a Carleson measure with the Carleson norm bounded by $\|f_2\|_{\bmo}$. Observe also that Lemma~\ref{main low frequency estim} yields that
\nm{kernelestforlowfreq}{
(P_{k_0}^0\circ T^{-\phase_0}_{d_0}\circ M_{\mathfrak m}) (f_0) = (P_{k_0}^0\circ T^{-\phase_0}_{d_0} \circ P^0_k \circ M_{\mathfrak m}) (f_0)
= K * ((P^0_k \circ M_{\mathfrak m}) (f_0))
}
for $k \geq k_0$, with $|K(\cdot)| \lesssim \langle \cdot \rangle^{-n-\varepsilon}$.

Therefore using Minkowski's integral inequality and estimate \eqref{carleson estim}, the first factor on the right-hand side of \eqref{eq:lowfreq11} can be controlled using the non-tangential maximal function as
\[
\|f_2\|_{\bmo}\int \abs{K(z)}\left(\int \sup_{k\geq k_0, |y-x| \lesssim 2^{-k}} 
 |P_{k}^0 (M_{\mathfrak m} f_0)(y-z)|^2 \dd x\right)^{1/2}\, \dd z.
\]
 However, since $P_k$ is convolution with a Schwartz function scaled by a factor $2^{-k}$ and $M_\mathfrak{m}$ is uniformly bounded, we have that
$$\sup_{k\geq k_0, |y-x| \lesssim2^{- k}} 
 |P_{k}^0 (M_{\mathfrak m} f_0)(y-z)|\lesssim |(M f_0)(x-z)|,$$ where $M$ is the Hardy-Littlewood maximal function.
Thus we have the estimate
\begin{equation}\label{ineq:firstfactor}
\brkt{\sum_{k\geq k_0}
\int \abs{P_{k_0}^0\circ T^{-\phase_0}_{d_0} (M_{\mathfrak m} f_0) \, (Q_{k}^{u_2} \circ T^{\phase_2}_{d_2})(f_2)}^2 \dd x}^{1/2}
\lesssim \norm{f_0}_{L^2}\norm{f_2}_\bmo
\end{equation}
for the first factor in \eqref{eq:lowfreq11}.

We will see that to estimate \eqref{highfreq11} uniformly in $k$ is a similar task to that done in Section~\ref{sec:sigmaj}. We must, however, also obtain summability in $k$. The content of the next lemma is the observation that there is some decay in size of the Carleson norms that appear.

\begin{lem}\label{lem:smallcarlnorm}
If $n \geq 2$ and $f \in \bmo$ then, for $j=1,2$,
\m{
\dd \mu_k(x,t) = \sum_{k'=0}^\infty |(Q_{k+k'}^{u_j} \circ T^{\phase_1}_{d_j})(f)(x)|^2 \delta_{2^{-k'}}(t) \dd x
}
is a Carleson measure with Carleson norm $2^{-k/2}\norm{f}_{\bmo}^2$.
\end{lem}
\begin{proof}
For definiteness take $j=1$. Since we can write $Q_{k+\ell}^{u_1} \circ T^{\phase_1}_{b_1} = Q_{k+\ell}^{u_1} \circ T^{\phase_1}_{b_1} \circ \tilde{Q}_{k+\ell}$ where $\tilde{Q}_{k+\ell}$ maps $\bmo$ into $L^\infty$ uniformly in $k+\ell$, as a first step we consider $f\in L^\infty$.

The operator $Q_{k+\ell}^{u_1} \circ T^{\phase_1}_{b_1}$ is just the $(k+\ell)$-th component of the Seeger-Sogge-Stein decomposition of the Fourier integral operator $T^{\phase_1}_{b_1}$, which we saw in Section \ref{sec:smallballs}. This in turn is split into $O(2^{(k+\ell)(n-1)/2})$ separate operators $T_{k+\ell}^\nu$ ($\nu = 1,2,\dots,c2^{(k+\ell)(n-1)/2}$) with kernels $K_{k+\ell}^\nu(x,y)$ which, as can be seen from \eqref{eq:claimkernel}, satisfy
\nm{ineq:secondkernelest}{
|K_{j}^\nu(x,y)|\leq c2^j\{1 + 2^j|(x + \nabla\phase_1(\xi_j^\nu) - y)_1| + 2^{j/2}|(x + \nabla\phase_1(\xi_j^\nu) - y)'|\}^{-\NNN},
}
for any $\NNN>0$ and all $j\geq 0$. Here we have chosen a coordinate system where $x_1$ is parallel to $\xi_j^\nu$ (which was also defined in Section \ref{sec:smallballs}) and $x'$ denotes the vector of remaining coordinates. For a given ball $B \subset \R^n$ with centre $x_0$ and radius $r\leq 1$ we write $g^\nu_j = f_1\chi_{\RR^\nu_j}$ and $h^\nu_j = f_1\chi_{(\RR^\nu_j)^c}$, with $\RR^\nu_j$ being a rectangle with side-length $2r$ parallel to $\nabla\phase_1(\xi_j^\nu)$, side-length $2r^{1/2}$ in the remaining directions and centre $x_0 + \nabla\phase_1(\xi_j^\nu)$. Clearly then $f_1 = g^\nu_j + h^\nu_j$ and
\m{
(Q_{j}^{u_1} \circ T^{\phase_1}_{b_1})(f_1) = \sum_\nu T_{j}^\nu(f_1) = \sum_\nu T_{j}^\nu(g^\nu_j) + \sum_\nu T_{j}^\nu(h^\nu_j).
}

Since $T^\nu_j$ are multipliers whose $L^2$-norms are bounded by $2^{-j(n-1)/2}$ and whose symbols have almost disjoint support, i.e.~with finitely many overlaps, we have
\ma{
\int_B \left|\sum_\nu T_{j}^\nu(g^\nu_j)(x) \right|^2 \dd x
&\leq \int \left|\sum_\nu T_{j}^\nu(g^\nu_j)(x) \right|^2 \dd x
\lesssim \sum_\nu \int \left|T_{j}^\nu(g^\nu_j)(x) \right|^2 \dd x \\
&\lesssim \sum_\nu 2^{-j(n-1)}\int \left|g^\nu_j(x) \right|^2 \dd x
\lesssim \sum_\nu 2^{-j(n-1)}\int_{R^\nu_j} \left|f_1(x)\right|^2 \dd x  \\
&\lesssim \sum_\nu 2^{-j(n-1)}|\RR^\nu_j| \norm{f_1}_{L^\infty}^2
\lesssim 2^{-j(n-1)/2}r^{-(n-1)/2}|B| \norm{f_1}_{L^\infty}^2.
}
Using \eqref{ineq:secondkernelest}, we also have
\ma{
&|T_{j}^\nu(h^\nu_j)(x)| \\
&\leq \int_{(\RR^\nu_j)^c} \frac{c2^jf_1(y)}{\{1 + 2^j|(x + \nabla\phase_1(\xi_j^\nu) - y)_1| + 2^{j/2}|(x + \nabla\phase_1(\xi_j^\nu) - y)'|\}^{n+1}} \dd y \\
&\leq 2^{-j(n-1)/2}\int_{(\RR^\nu_j)^c} \frac{c2^{j(n+1)/2}f_1(y)}{\{1 + 2^j|(x + \nabla\phase_1(\xi_j^\nu) - y)_1| + 2^{j/2}|(x + \nabla\phase_1(\xi_j^\nu) - y)'|\}^{n+1}} \dd y \\
&\leq 2^{-j(n-1)/2}\norm{f_1}_{L^\infty}\int_{(\RR^\nu_j)^c} \frac{c2^{j(n+1)/2}}{\{1 + 2^j|(x + \nabla\phase_1(\xi_j^\nu) - y)_1| + 2^{j/2}|(x + \nabla\phase_1(\xi_j^\nu) - y)'|\}^{n+1}} \dd y.
}
For $x \in B$ and $y \in (\RR^\nu_j)^c$ we must have that either
\begin{equation}\label{main denominator estim}
    2^j|(x + \nabla\phase_1(\xi_j^\nu) - y)_1| \geq 2^jr \quad \mbox{or} \quad 2^{j/2}|(x + \nabla\phase_1(\xi_j^\nu) - y)'| \geq 2^{j/2}r^{1/2}.
\end{equation}
Moreover, for those $j$ such that $2^{-j} \leq r$, we have that $2^jr \geq 2^{j/2}r^{1/2}$. Thus, for all such $j$, \eqref{main denominator estim} yields that
\ma{
&1 + 2^j|(x + \nabla\phase_1(\xi_j^\nu) - y)_1| + 2^{j/2}|(x + \nabla\phase_1(\xi_j^\nu) - y)'| \\
&\geq \frac{1}{2}(2^{j/2}r^{1/2} + 2^j|(x + \nabla\phase_1(\xi_j^\nu) - y)_1| + 2^{j/2}|(x + \nabla\phase_1(\xi_j^\nu) - y)'|).
}

Therefore
\ma{
&\int_{(\RR^\nu_j)^c} \frac{c2^{j(n+1)/2}}{\{1 + 2^j|(x + \nabla\phase_1(\xi_j^\nu) - y)_1| + 2^{j/2}|(x + \nabla\phase_1(\xi_j^\nu) - y)'|\}^{n+1}} \dd y \\
& \lesssim \int_{(\RR^\nu_j)^c} \frac{c2^{j(n+1)/2}}{\{2^{j/2}r^{1/2} + 2^j|(x + \nabla\phase_1(\xi_j^\nu) - y)_1| + 2^{j/2}|(x + \nabla\phase_1(\xi_j^\nu) - y)'|\}^{n+1}} \dd y \\
& \lesssim \int \frac{c2^{j(n+1)/2}}{\{2^{j/2}r^{1/2} + 2^j|y_1| + 2^{j/2}|y'|\}^{n+1}} \dd y \\
&\leq \int \frac{c2^{jn/2}}{\{2^{j/2}r^{1/2} + 2^{j/2}|y|\}^{n+1}} \dd y \\
&\leq \int \frac{c2^{-j/2}}{\{r^{1/2} + |y|\}^{n+1}} dy \leq \frac{2^{-j/2}}{r^{1/2}}\int \frac{cr^{1/2}}{\{r^{1/2} + |y|\}^{n+1}} \dd y \lesssim \frac{2^{-j/2}}{r^{1/2}}.
}
We conclude that for $x \in B$ and $j$ such that $2^{-j}  \leq r$,
\ma{
|T_{j}^\nu(h^\nu_j)(x)| \lesssim 2^{-j(n-1)/2}\frac{2^{-j/2}}{r^{1/2}}\norm{f_1}_{L^\infty}
}
Combining these estimates enables us to estimate
\ma{
    \int_{B\times [0,r]} |\dd\mu_k(x,t)| &= \sum_{2^{-\ell} \leq r} \int_B |(Q_{k+\ell}^{u_1} \circ T^{\phase_1}_{b_1}\circ \tilde{Q}_{k+\ell})(f_1)|^2 \dd x \\
&\lesssim \sum_{2^{-\ell} \leq r}\left(2^{-k(n-1)/2}2^{-\ell(n-1)/2}r^{-(n-1)/2} + 2^{-k/2}2^{-\ell/2}r^{-1/2}\right)|B| \norm{\tilde{Q}_{k+\ell}(f_1)}_{L^\infty}^2 \\
&\lesssim \left(2^{-k(n-1)/2} + 2^{-k/2}\right)|B| \norm{f_1}_{\bmo}^2 \\
&\lesssim 2^{-k/2}|B| \norm{f_1}_{\bmo}^2
}
for $k \geq 0$ and $n\geq2$. Thus we have even proved $d\mu_k(x,\ell)$ is a Carleson measure with norm at most $2^{-k/2}\norm{f_1}_{\bmo}^2$ provided $n\geq 2$.
\end{proof}


\subsubsection{$\bmo \times \dots \times \bmo \times L^2 \to L^2$} Here we take $n\geq 2,$ $m_j= -(n-1)/2$, $f_j \in \bmo$ 
for $j=1,\ldots, N-1$, $m_N=0$ and $f_N \in L^2$. 

Lemma \ref{lem:Wase-Lin} shows us that
\begin{equation}\label{eq:uniform_bound_2}
     \sup_{k\geq k_0}\norm{(P_{k}^{u_j} \circ T^{\phase_j}_{d_j})(f_j)}_{L^\infty} \lesssim  \norm{f_j}_{\bmo} \quad \mbox{for $j=3,\dots,N$ whenever $p_j = \infty$.}
\end{equation}
Using \eqref{eq:uniform_bound_2}, \eqref{ineq:carll2} and Theorem~\ref{linearhpthm} we can estimate
\ma{
&\brkt{\sum_{k\geq k_0}
\int \abs{(Q_{k}^{u_1} \circ T^{\phase_1}_{d_1})(f_1)\,\prod_{j=3}^N(P_{k}^{u_j} \circ T^{\phase_j}_{d_j})(f_j)}^2\dd x}^{1/2} \\
&\lesssim \norm{f_1}_{\bmo}\prod_{j=3}^{N-1} \norm{f_j}_\bmo \norm{f_N}_{L^2}
}
and combining this with \eqref{ineq:firstfactor} bounds \eqref{eq:lowfreq11}, as required.

To bound \eqref{highfreq11}, we see from Lemma~\ref{lem:smallcarlnorm} and \eqref{eq:uniform_bound_2} that
\[
    \sum_{k\geq k_0} \abs{(Q_{k+\ell}^{u_1} \circ T^{\phase_1}_{d_1})(f_1)\,(Q_{k+\ell}^{u_2} \circ T^{\phase_2}_{d_2})(f_2)\,\prod_{j=3}^{N-1}(P_{k}^{u_j} \circ T^{\phase_j}_{d_j})(f_j)}^2\dd x\, \delta_{2^{-k}}(t)
\]
is a Carleson measure with norm $2^{k/2}\prod_{j=2}^{N-1}\norm{f_j}_\bmo^2$. Therefore, again by \eqref{ineq:carll2} and Theorem~\ref{linearhpthm}, we see that \eqref{highfreq11} is bounded by $2^{k/4}\norm{f_0}_{L^2}\prod_{j=1}^{N-1}\norm{f_j}_\bmo\norm{f_N}_{L^2}$, which again is sufficient for our purposes.

\subsubsection{$L^2 \times \bmo \times \dots \times \bmo \to L^2$} Here we take $m_1=0$ and $f_1 \in L^2$ and $m_j= -(n-1)/2$, $f_j \in \bmo$ 
for $j=2,\ldots, N$.

Using \eqref{eq:uniform_bound_2}, \eqref{ineq:quadraticestimate} and Theorem~\ref{linearhpthm} we can estimate
\m{
\brkt{\sum_{k\geq k_0}
\int \abs{(Q_{k}^{u_1} \circ T^{\phase_1}_{d_1})(f_1)\,\prod_{j=3}^N(P_{k}^{u_j} \circ T^{\phase_j}_{d_j})(f_j)}^2\dd x}^{1/2}
\lesssim \norm{f_0}_{L^2}\norm{f_1}_{L^2}\prod_{j=3}^N \norm{f_j}_\bmo
}
and combining this with \eqref{ineq:firstfactor} bounds \eqref{eq:lowfreq11}, as required.

To bound \eqref{highfreq11}, we see from Lemma~\ref{lem:smallcarlnorm} and \eqref{eq:uniform_bound_2} that
\[
    \sum_{k\geq k_0} \abs{(Q_{k+\ell}^{u_2} \circ T^{\phase_2}_{d_2})(f_2)\,\prod_{j=3}^N(P_{k}^{u_j} \circ T^{\phase_j}_{d_j})(f_j)}^2\dd x\, \delta_{2^{-k}}(t)
\]
is a Carleson measure with norm $2^{k/2}\prod_{j=2}^N\norm{f_j}_\bmo^2$. Therefore, by \eqref{ineq:carll2} and Theorem~\ref{linearhpthm}, we see that \eqref{highfreq11} is bounded by $2^{k/4}\norm{f_0}_{L^2}\norm{f_1}_{L^2}\prod_{j=2}^N\norm{f_j}_\bmo$, which is sufficient to conclude the proof of this endpoint case.

\subsection{The endpoint cases with target space $h^1$}

The operator $M_{\mathfrak{m}}$ (which we recall depends on $k$) can be viewed as a pseudodifferential operator and therefore (see \cite{Gol})
\[
    \norm{M_{\mathfrak{m}}(f_0)}_{\bmo} \lesssim \sum_{\abs{\alpha}\lesssim 1} \norm{\partial^\alpha {\mathfrak{m}}}_{L^\infty} \norm{f_0}_\bmo \lesssim \norm{f_0}_\bmo,
\]
with implicit constants independent of $k$. Thus Lemma \ref{lem:Wase-Lin}  yields
\begin{equation}\label{linftyonmk}
 \sup_{k\geq k_0}\norm{P_{k}^0\circ T^{-\phase_0}_{d_0} (M_{\mathfrak m} f_0)}_{L^\infty}\lesssim \norm{f_0}_{\bmo} \lesssim 1.
\end{equation}
Moreover, as a scholium to Lemma \ref{lem:Wase-Lin}, we have that
\begin{equation}\label{eq:uniform_bound_3}
     \sup_{k\geq k_0}\norm{(Q_{k}^{u_j} \circ T^{\phase_j}_{d_j})(f_j)}_{L^\infty} \lesssim  \norm{f_j}_{\bmo} \quad \mbox{for $j=1,2$ when $p_j = \infty$.}
\end{equation}

\subsubsection{$\bmo \times \dots \times \bmo \times h^1 \to h^1$} Here we take $m_j= -(n-1)/2$ for $j=0, \dots, N$, $f_j \in \bmo$ for $j=0, \dots, N-1,$ and $f_N\in h^1.$

By \eqref{eq:uniform_bound_2} and \eqref{linftyonmk} we see that
\ma{
    \dd\mu(x,t)&:=\sum_{k\geq k_0}P_{k}^0\circ T^{-\phase_0}_{d_0} (M_{\mathfrak m} f_0) \times \\ &\left[(Q_{k}^{u_1} \circ T^{\phase_1}_{d_1})(f_1)\, (Q_{k}^{u_1} \circ T^{\phase_2}_{d_2})(f_2)\,\prod_{j=3}^{N-1}(P_{k}^{u_j} \circ T^{\phase_j}_{d_j})(f_j)\right] \dd x\, \delta_{2^{-k}}(t)
}
defines a Carleson measure with Carleson norm bounded by 
\[
    \norm{f_0}_{\bmo}\prod_{j\neq j_0} \norm{f_j}_{\bmo}
    \norm{f_{1}}_{\bmo}\norm{f_2}_{\bmo}.
\]
So \eqref{eq:Lu-Be} becomes
\[
    \sum_{k\geq k_0} \int  P_{k}^{u_{N}} T_{d_N}^{\phase_N}(f_N ) (x) \dd \mu(x,2^{-k}),
\]
and arguing as in Section \ref{sec:bmoh1h1}, it follows that 
\[
   \abs{\sum_{k\geq k_0} \int  P_{k}^{u_{N}}\circ T_{d_N}^{\phase_N}(f_N)(x) \dd \mu(x,2^{-k})}\lesssim 
  \prod_{j=0}^{N-1} \norm{f_j}_{\bmo}
    \norm{f_N}_{h^1}.
\]

\subsubsection{$h^1 \times \bmo \times \dots \times \bmo \to h^1$} Here we take $m_j= -(n-1)/2$ for $j=0, \dots, N$, $f_0\in \bmo$, $f_1\in h^1$ and $f_j\in \bmo$ for $j=2, \dots, N.$

Using the estimates \eqref{eq:uniform_bound_2} and \eqref{linftyonmk} again together with Proposition \ref{cor:monster 1}, Theorem \ref{linearhpthm} and \eqref{defn:bmo} mean we can estimate \eqref{eq:Lu-Be} by
\[
\begin{split}
&\norm{f_0}_{\bmo}\prod_{j=3}^N \norm{f_j}_{\bmo}
    \norm{T^{\phase_1}_{d_1}(f_1)}_{H^1}\norm{T^{\phase_2}_{d_2}(f_2)}_{\mathrm{BMO}} \\
    &\leq \norm{f_0}_{\bmo}\prod_{j=3}^N \norm{f_j}_{\bmo}
    \norm{f_1}_{h^1}\norm{f_2}_{\bmo}.
\end{split}
\]
\subsubsection{$L^2 \times L^2 \times \bmo \times \dots \times \bmo \to h^1$}
We choose $m_1=m_2=0$, $f_1,f_2\in L^2$, $m_j=-\frac{n-1}{2}$, $f_j\in\bmo$ for $j=3,\ldots N$. 

Once again, \eqref{eq:uniform_bound_2}, \eqref{linftyonmk} and Theorem \ref{linearhpthm}, this time together with \eqref{ineq:quadraticestimate} mean we can estimate \eqref{eq:Lu-Be} by
\ma{
&\sum_{k\geq k_0}
\int \abs{(Q_{k}^{u_1} \circ T^{\phase_1}_{d_1})(f_1)\, (Q_{k}^{u_2} \circ T^{\phase_2}_{d_2})(f_2)}\dd x\prod_{j=3}^N \norm{f_j}_{\bmo} \\
&\lesssim \brkt{\sum_{k\geq k_0}
\int \abs{(Q_{k}^{u_1} \circ T^{\phase_1}_{d_1})(f_1)}^2 \dd x}^{1/2}\brkt{\sum_{k\geq k_0}
\int \abs{(Q_{k}^{u_2} \circ T^{\phase_2}_{d_2})(f_2)}^2 \dd x}^{1/2}\prod_{j=3}^N \norm{f_j}_{\bmo} \\
&\lesssim \norm{f_1}_{L^2}\norm{f_2}_{L^2}\prod_{j=3}^N \norm{f_j}_{\bmo},
}
where we have also used the Cauchy-Schwarz inequality and quadratic estimates. 

\subsubsection{$\bmo \times \bmo \times \dots \times \bmo \times L^2 \times L^2 \to h^1$}
We choose $m_j=-\frac{n-1}{2}$ and $f_j\in\bmo$ for $j=1,\ldots, {N-2}$, and $m_j=0$, $f_j\in L^2$ for $j=N,N-1$. 

Via \eqref{eq:uniform_bound_2} and \eqref{linftyonmk}
\m{
\sum_{k\geq k_0} P_{k}^0\circ T^{-\phase_0}_{d_0} (M_{\mathfrak m} f_0)\left[(Q_{k}^{u_1} \circ T^{\phase_1}_{d_1})(f_1)\, (Q_{k}^{u_2} \circ T^{\phase_2}_{d_2})(f_2)\right]\prod_{j=3}^{N-2}(P_{k}^{u_j} \circ T^{\phase_j}_{d_j})(f_j)\delta_{2^{-k}}(t)\dd x
}
can be seen to be a Carleson measure. Since when $m_j = 0$ for $j=N$ or $N-1$, $d_j$ is independent of $k$, \eqref{ineq:carll2} together with Theorem~\ref{linearhpthm} can be used to estimate \eqref{eq:Lu-Be} by
\m{\prod_{j=1}^{N-2} \norm{f_j}_{\bmo} \norm{T^{\phase_{N-1}}_{d_{N-1}}(f_{N-1})}_{L^2}\norm{T^{\phase_{N}}_{d_{N}}(f_{N})}_{L^2} \lesssim \prod_{j=1}^{N-2} \norm{f_j}_{\bmo} \norm{f_{N-1}}_{L^2}\norm{f_N}_{L^2}.
}

\subsubsection{$L^2 \times \bmo \times \bmo \times \dots \times \bmo \times L^2 \to h^1$}

We choose $m_1=m_N=0$, $f_1,\, f_N\in L^2$, $m_j=-\frac{n-1}{2}$ and $f_j\in\bmo$ for $j=2,\ldots, N-1$.

This time we again first apply \eqref{eq:uniform_bound_2} and \eqref{linftyonmk} to \eqref{eq:Lu-Be} but then the Cauchy-Schwartz inequality to obtain the estimate
\ma{
&\norm{f_0}_{\bmo} \prod_{j=3}^{N-1} \norm{f_j}_{\bmo}  \brkt{\sum_{k\geq k_0} \int \abs{(Q_{k}^{u_{2}} \circ T^{\phase_2}_{d_{2}})(f_{2})(P_{k}^{u_N} \circ T^{\phase_N}_{d_N})(f_N)}^2 \dd x}^{1/2} \times\\
&\qquad \brkt{\sum_{k\geq k_0}
\int \abs{(Q_{k}^{u_{1}} \circ T^{\phase_1}_{d_{1}})(f_{1})}^2 \dd x}^{1/2}
}
Thereafter, \eqref{ineq:carll2}, \eqref{ineq:quadraticestimate} and Theorem~\ref{linearhpthm} lead us to the bound
\m{
\norm{f_0}_{\bmo} \norm{f_1}_{L^2}\prod_{j=2}^{N-1} \norm{f_j}_{\bmo} \norm{f_N}_{L^2}.
}

\subsection{The endpoint case with target space $\bmo$} \label{sec:sigma12bmotarget} Here we take $m_j= -(n-1)/2$, and $f_j \in \bmo$ for $j=1, \dots, N$.

Just as we did in the proof of Lemma \ref{lem:Wase-Lin}, and with the same notation, we write
\nma{eq:commutator}{
P_{k}^0\circ T^{-\phase_0}_{d_0} &= P_{k}^0\circ T^{-\phase_0}_{d^\flat_0} + P_{k}^0\circ T^{-\phase_0}_{d^\sharp_0} \\
&= P_{k}^0\circ T^{-\phase_0}_{d^\flat_0} + \sum_{j=k_0}^k 2^{(k-j)m_0}Q_{j}\circ T^{-\phase_0}_{\gamma}
}
with the help of \eqref{defn: Rk}.

To estimate the term arising from the sum in $j$ in \eqref{eq:commutator} we argue as in Section \ref{sec:carlsmallnorm} and are led to the expression
\m{
\sum_{k=0}^\infty 2^{km_0} \sum_{\ell=k_0}^\infty \int Q_\ell\circ T^{-\phase_0}_{\gamma} (M_k f_0) \, (Q_{k+\ell}^{u_1} \circ T^{\phase_1}_{d_1})(f_1)\, (Q_{k+\ell}^{u_1} \circ T^{\phase_2}_{d_2})(f_2) \,\prod_{j=3}^N(P_{k+\ell}^{u_j} \circ T^{\phase_j}_{d_j})(f_j)\dd x.
}
The sum in $\ell$ can be estimated using \eqref{eq:uniform_bound_2}, the fact (from Lemma \ref{lem:smallcarlnorm}) that
\m{
\sum_{\ell\geq k_0} \abs{(Q_{k+\ell}^{u_1} \circ T^{\phase_1}_{d_1})(f_1)(x)\, (Q_{k+\ell}^{u_1} \circ T^{\phase_2}_{d_2})(f_2)(x)}\dd x\, \delta_{2^{-\ell}}(t)
}
is a Carleson measure with Carleson norm of size $2^{-k/2}\norm{f_1}_\bmo\norm{f_2}_\bmo$, and \eqref{ineq:carlh1}. It is then straight-forward to sum in $k$.

To deal with the first term of the right-hand side of \eqref{eq:commutator} we write
\ma{
P_{k}^0\circ T^{-\phase_0}_{d^\flat_0}\circ M_{\mathfrak m} &=  T^{-\phase_0}_{d^\flat_0}\circ P_{k}^0\circ M_{\mathfrak m} \\
&= T^{-\phase_0}_{d^\flat_0}\circ [P_{k}^0, M_{\mathfrak m}] + T^{-\phase_0}_{d^\flat_0}\circ M_{\mathfrak m}\circ P_{k}^0
}
A fairly standard calculation shows that the kernel of $[P_{k}^0, M_{\mathfrak m}]$ is integrable and of size $2^{-k}$. This combined with the estimate of the kernel of $T^{-\phase_0}_{d^\flat_0}$ from Lemma \ref{main low frequency estim} shows that
\m{
\norm{T^{-\phase_0}_{d^\flat_0}\circ [P_{k}^0, M_{\mathfrak m}](f_0)}_{L^1} \lesssim 2^{-k}\norm{f_0}_{L^1} \lesssim 2^{-k}\norm{f_0}_{h^1}
}
and so, together with \eqref{eq:uniform_bound_2} and \eqref{eq:uniform_bound_3}, this proves
\ma{
&\sum_{k=k_0}^\infty \int T^{-\phase_0}_{d^\flat_0}\circ [P_{k}^0, M_{\mathfrak m}](f_0) \, (Q_{k}^{u_1} \circ T^{\phase_1}_{d_1})(f_1)\, (Q_{k}^{u_1} \circ T^{\phase_2}_{d_2})(f_2) \,\prod_{j=3}^N(P_{k}^{u_j} \circ T^{\phase_j}_{d_j})(f_j)\dd x \\
\quad &\lesssim \norm{f_0}_{h^1}\prod_{j=1}^N\norm{f_j}_{\bmo}.
}

Finally, the term associated with $T^{-\phase_0}_{d^\flat_0}\circ M_{\mathfrak m}\circ P_{k}^0$ can be dealt with by first writing
\nma{theend}{
&\sum_{k=k_0}^\infty \int T^{-\phase_0}_{d^\flat_0} \circ M_{\mathfrak m} \circ P_{k}^0 (f_0) \, (Q_{k}^{u_1} \circ T^{\phase_1}_{d_1})(f_1)\, (Q_{k}^{u_1} \circ T^{\phase_2}_{d_2})(f_2) \,\prod_{j=3}^N(P_{k}^{u_j} \circ T^{\phase_j}_{d_j})(f_j)\dd x \\
\quad &= \sum_{k=k_0}^\infty \int M_{\mathfrak m} \circ P_{k}^0 (f_0) \, T^{\phase_0}_{d^\flat_0}\brkt{(Q_{k}^{u_1} \circ T^{\phase_1}_{d_1})(f_1)\, (Q_{k}^{u_1} \circ T^{\phase_2}_{d_2})(f_2) \,\prod_{j=3}^N(P_{k}^{u_j} \circ T^{\phase_j}_{d_j})(f_j)}\dd x.
}
The kernel estimate of $T^{-\phase_0}_{d^\flat_0}$ from Lemma \ref{main low frequency estim} shows that since
\m{
\sum_{\ell\geq k_0} \abs{(Q_{k}^{u_1} \circ T^{\phase_1}_{d_1})(f_1)\, (Q_{k}^{u_1} \circ T^{\phase_2}_{d_2})(f_2) \,\prod_{j=3}^N(P_{k}^{u_j} \circ T^{\phase_j}_{d_j})(f_j)}\dd x\, \delta_{2^{-\ell}}(t)
}
is a Carleson measure, then even
\m{
\sum_{\ell\geq k_0} \abs{T^{\phase_0}_{d^\flat_0}\brkt{(Q_{k}^{u_1} \circ T^{\phase_1}_{d_1})(f_1)\, (Q_{k}^{u_1} \circ T^{\phase_2}_{d_2})(f_2) \,\prod_{j=3}^N(P_{k}^{u_j} \circ T^{\phase_j}_{d_j})(f_j)}}\dd x\, \delta_{2^{-\ell}}(t)
}
is a Carleson measure. Therefore applying the uniform bound of $\mathfrak m$ and \eqref{ineq:carlh1} in \eqref{theend} completes the proof.

\section{Boundedness results for $T_{\sigma_{0}}^\Phi$}\label{sec:sigma0}
For the case of $T^\Phase_{\sigma_{0}}$ given by
\m{
T^\Phase_{\sigma_{0}} (f_1,\dots,f_N)(x) = \int_{\R^{nN}} \sigma_0 (x,\Xi)\prod_{j=1}^N\left(\widehat{f}_j(\xi_j)e^{ix\cdot\xi_j}\right) e^{i\Phase(\Xi)}\ddd \Xi.
}
 we use a separation of variables technique as follows.
 
Let $Q$ be a closed cube in $\R^{nN}$ of side-length $L$ which compactly contains the $\Xi$-support of $\sigma_0$. We extend $\sigma_0 (x,\Xi)|_{\Xi\in Q}$ periodically in the $\Xi$-variables with period $L$ to $\widetilde{\sigma_0}(x,\Xi)\in \mathcal{C}^{\infty}(\R_x^n\times\R^{nN}_\Xi).$ Let $\zeta\in \mathcal{C}^{\infty}_c(\R^{nN})$ with $\supp \zeta \subset Q$ and $\zeta=1$ on $\Xi$-support of $\sigma_0 (x,\Xi)$, so we have $ \sigma_0 (x,\Xi)=\widetilde{\sigma_0}(x,\Xi)\zeta(\Xi)$. We can then find the Fourier series coefficients of $\widetilde{\sigma}(x,\Xi) $:
\begin{align*}\label{estim ak1}
a_K(x) &= \widehat{\widetilde{\sigma_0}(x,\Xi)}(x, K) = \frac1{L^n}\int_Q e^{-i\frac{2\pi}L \Xi\cdot K}\widetilde{\sigma_0 }(x,\Xi)\dd\Xi \\
&=  \frac1{L^n}\int_{\R^{nN}} e^{-i\frac{2\pi}L \Xi\cdot K}\sigma_0 (x,\Xi)\dd\Xi,
\end{align*}
where $\Xi=(\xi_1, \ldots, \xi_N)\in \R^{nN}$, $K=(k_1, \ldots, k_N)\in \R^{nN}$ and $\Xi\cdot K= \sum_{j=1}^{N}\xi_j \cdot k_j = \sum_{j=1}^{N} \sum_{\ell=1}^{n}\xi^{\ell}_j k^{\ell}_j .$ Also observe that using this notation one has that $|k_j|^2= \sum_{\ell=1}^{n}(k^{\ell}_j)^2$. Integration by parts then yields that
\begin{equation*}\label{estim ak2}
|\partial^{\alpha} a_{K}(x)| = \frac{c_{n,M, L}}{|k_j^{\ell}|^{2M}}\left|\int e^{-i\frac{2\pi}L\Xi\cdot K}\partial_{\xi_{j}^\ell}^{2M} \partial_{x}^{\alpha}\sigma_0 (x,\Xi)\dd\Xi\right|
\end{equation*}
for all multi-indices $\alpha$, any $M \geq 0$, and some constants $c_{n,M,L}$. Therefore, the boundedness of the $\Xi$-support of $\sigma_0 (x,\Xi)$ and the fact that $|\partial^{\alpha}a_{K}(x)|\lesssim 1$ imply that 
\begin{equation}\label{derivative of a_K}
 |\partial^{\alpha}a_K(x)|\lesssim (1+\sum_{j=1}^{N} |k_j|^2)^{-M}   
\end{equation} for all $x\in \R^n$ and $M\geq 0$.

We now choose $\theta \in \mathcal{C}_c^{\infty}(\R^n)$ such that $1= \prod_{j=1}^N \theta(\xi_j)$ for $\Xi = (\xi_1,\dots,\xi_N)$ on the support of $\zeta$. We have then even
\begin{equation*}
    1= \theta((\xi_1+\dots +\xi_N)/\sqrt{N})\prod_{j=1}^N \theta(\xi_j)
\end{equation*}
for $\Xi = (\xi_1,\dots,\xi_N)$ on the support of $\zeta$. Using the Fourier expansion of $\widetilde{\sigma_0}(x,\Xi)$, we can write
\[
T^\Phase_{\sigma_{0}} (f_1,\dots,f_N)(x) = \sum_{K\in\Z^{nN}} a_K(x)T_{\theta(\cdot/\sqrt{N})}^{\phase_0}\left(\prod_{j=1}^N T_{\theta}^{\phase_j}\circ\tau_{\frac{2\pi k_j}{L}}(f_j)\right)(x),
\]
where $\tau_h f(x) := f(x-h)$.

Since we only need consider the endpoint cases of Corollary \ref{cor:endpointcases}, the analysis is confined to the spaces $h^1$, $L^2$ and $\bmo$. Now observe that since $\theta\in\mathcal{C}_c^{\infty}(\R^n)$, Lemma \ref{main low frequency estim}, \eqref{eq:local_H1} and \eqref{defn:bmo} yield that
\begin{equation*}
    \norm{T_{\theta}^{\phase_j}(f)}_{L^p} \lesssim \norm{f}_{X^p} \quad \mbox{and} \quad \norm{T_{\theta(\cdot/\sqrt{N})}^{\phase_0}(f)}_{X^p} \lesssim \norm{f}_{L^p}
\end{equation*}
for $p = 1,2,\infty$. Combining these estimates with the translation invariance of the norms and H\"older's inequality gives 
\[
    \norm{T_{\theta(\cdot/\sqrt{N})}^{\phase_0}\left(\prod_{j=1}^N T_{\theta}^{\phase_j}\circ\tau_{\frac{2\pi k_j}{L}}(f_j)\right)}_{X^{p_0}}\lesssim   \prod_{j=1}^N \norm{f_j}_{X^{p_j}}. 
\]
for all the endpoint cases of $p_0,p_1,\dots,p_N$ in Corollary \ref{cor:endpointcases}. Finally, the boundedness of $T_{\sigma_0}^{\phase_0}$ follows by applying \eqref{derivative of a_K} with the inclusions $\mathcal{C}^1_b\cdot h^1\subseteq h^1$, $L^\infty \cdot L^2 \subseteq L^2$ and $\mathcal{C}^1_b\cdot \bmo \subseteq \bmo$ (see \cite{Gol}).

\end{document}